\documentclass[reqno,12pt]{amsart}
\usepackage{color}
\usepackage{amsfonts}
\usepackage{amsmath}
\usepackage{amssymb}
\newcommand{\intT}{\int_0^T}
\newcommand{\intt}{\int_0^t}
\newcommand{\DASP}{({\rm dom}A^*)'}

\hoffset= - 0.4 cm

\newcommand{\zl}{\lambda}
\newcommand{\zreal}{\Re{\textstyle e}\,}
\newcommand{\ZR}{\rangle}
\newcommand{\ZL}{\langle}

\newcommand{\ZSI}{\sigma}
\newcommand{\ZEP}{\epsilon}
\newcommand{\ZLA}{\label}

\newcommand{\ZD}{\;\mbox{\rm d}}
\newcommand{\ZOM}{\omega}
\newtheorem{Hypothesis}{Hypothesis}[section]
 \newtheorem{Theorem}{Theorem}[section]
\newtheorem{Proposition}[Theorem]{Proposition}
\newtheorem{Lemma}[Theorem]{Lemma}
\newtheorem{Corollary}[Theorem]{Corollary}

\newtheorem{Remark}[Theorem]{Remark}
\newtheorem{Definition}[Theorem]{Definition}

\def\R{{\mathbb R}}
\def\C{{\mathbb C}}

\def\A{\mathcal A}

\begin{document}

\title  [LOI, NCVE and boundary controls]
 {Linear Operator
 Inequality and  Null Controllability with Vanishing Energy
 for unbounded control  systems}


\maketitle

\begin{center}

Luciano Pandolfi \footnote{Supported in part by Italian MURST and by
the project ``Groupement de Recherche en Contr\^ole des EDP entre la
France et l'Italie (CONEDP)''.}

\vspace{2mm}
        {\small \it Politecnico di Torino,
        Dipartimento di Matematica \par
        Corso Duca degli Abruzzi 24, 10129 Torino, Italy.\par
        e-mail \ \  luciano.pandolfi@polito.it }
\vspace{ 2 mm }

 Enrico Priola \footnote{Supported  by the
M.I.U.R. research project Prin 2008 ``Deterministic and stochastic
methods in the study of evolution problems''.}

\vspace{ 2 mm } {\small  \it Dipartimento di Matematica,
Universit\`a di Torino, \par
  via Carlo Alberto 10,  \ 10123, \ Torino, Italy. \par
 e-mail \ \ enrico.priola@unito.it }

\vspace{ 2 mm }

Jerzy Zabczyk \footnote{ Supported by the NSF grant DMS-0500270.}

\vspace{ 2 mm} {\small  \it Instytut Matematyczny,   Polskiej
Akademii Nauk, \par ul. Sniadeckich 8, 00-950, \ Warszawa,
Poland.\par
 e-mail \ \ zabczyk@impan.gov.pl }
\end{center}

{\vskip 5mm }


\par \ \par


\ \par

\noindent {\bf Abstract:} We consider    linear  systems  on a separable Hilbert space $H$, which are null
 controllable at some time $T_0>0$ under the action of a point or boundary
 control.
  Parabolic and hyperbolic control systems usually studied in applications are special cases.
 To every initial state $
y_0 \in H$
 we associate the minimal ``energy'' needed to transfer $
y_0 $ to $ 0 $ in a time $ T \ge T_0$ (``energy'' of a control being
the square of its $ L^2 $ norm).
We give both  necessary and sufficient conditions under which
 the minimal
 energy converges to $ 0 $ for $ T\to+\infty $. This
extends to boundary control systems the concept of {\em null controllability with vanishing energy\/} introduced by Priola and Zabczyk (Siam J. Control Optim.
42 (2003)) for distributed systems.
 The proofs in Priola-Zabczyk paper depend on    properties of the associated Riccati equation, which are not available in the present,
 general setting. Here we  base our results  on new properties of the quadratic regulator problem with stability and  the Linear Operator Inequality.

\section {Introduction  and preliminaries}

The paper~\cite{PZ} introduced  and studied the property of ``null
controllability with vanishing energy'', shortly NCVE, for systems
with distributed control action, which is as follows: consider a
semigroup control system
(cf.~\cite{BD1,CZ,LT1,LT2,PWX,TW,Za})
\[
\dot y=Ay+Bu,\qquad y(0)=y_0 \in H,
 \]
which is  null controllable in time $ T_0 >0$ (hence also for every
larger time $ T>T_0 $).  This null controllable
 system  is
NCVE when for every $ y_0 $ and $ \ZEP>0 $ there exist a time $ T $
and a control $ u $ which steers the initial state $y_0$
 to zero in
time $ T $ and, furthermore, its $ L^2(0,T; U) $-norm
 is less then $
\ZEP $.
 This concept has been already
 applied in some specific situations (see~\cite{Ic1,Ic2})
 and partially
 extended to the Banach space setting in~\cite{vanNeerve}.
Moreover, applications of NCVE property to Ornstein-Uhlenbeck
processes are given in~\cite{PZ2}.

The key result   in~\cite{PZ}, i.e., Theorem~1.1,  shows that, under
suitable properties on the operator $ A $ stated
 below, NCVE holds
if and only if the system is null controllable and furthermore the
spectrum of $ A $ is contained in the \emph{closed} half plane $\{
\zreal\zl \leq 0 \}$.

The goal of this paper is to extend this result to a large class of
boundary and point control systems (see Hypothesis 1.1), which essentially
includes all the classes of   systems whose null controllability has
been studied up to now. Our main results are Theorems~\ref{primo}
 and~\ref{secondo}.
  Moreover,
Corollary~\ref{coro:rema:nounif}
combines these results and gives a necessary and sufficient
conditions for NCVE,
which applies in  the cases   most frequently
encountered   in
applications.  Finally,  Section 3  provides  applications of our
main results.  In particular we establish NCVE for  boundary control problems
involving systems of  parabolic  equations recently considered in
\cite{CARA}.

The proofs that we give are based on ideas different  from those
in~\cite{PZ}. Moreover conditions imposed  for the sufficiency part
are weaker from those used in~\cite[Theorem~1.1]{PZ},  in the case
of distributed control systems.

Now we describe the notations and the class of
systems we are studying.

The main spaces in this paper are Hilbert, and are identified with their
duals unless explicitly stated.
The notations   are standard. For example,
 ${\mathcal L}(H,K)$ denotes the Banach space of all bounded linear
 operators from $H$ into $K$ endowed with the operator norm.

Let $ H $ be a Hilbert space with inner product $\langle \cdot,
\cdot \rangle $ and norm $|\cdot |$ and let $ A $ be a generator of
a $ C_0 $-semigroup on $ H $. Due to the fact that the spectrum of $
A $ has a   role in our arguments, we assume from the outset that $
H $ is a \emph{complex} Hilbert space.

Let $ A^* $ be the Hilbert space adjoint of $ A $.
Its domain with the graph norm
\[
| y|^2=\langle y,y\rangle+\langle A^* y,A^* y\rangle
 \]
is a Hilbert space \emph{which is not identified with its dual.
}
 It is well known
that $ \DASP $ (the dual of the Hilbert space ${\rm dom}\,A^*$) is a
Hilbert space and
\[
({\rm dom}\,A^*)\subset H=H'\subset \DASP
 \]
(with continuous and dense injections). Moreover,   $ A $
 admits an
extension $ \A $  to $\DASP   $, which generates a $ C_0 $-semigroup
$e^{\A t}$ on $ \DASP $. The domain of such extension is equal to $
H $ (see~\cite[Section 0.3]{LT1}, \cite[Chapter
 3]{BD1} and~\cite{TW}; see also
 Appendix~\ref{AppBASICsetting}).

The norm in $ \DASP $ is denoted by  $ |\cdot|_{-1} $,
  and it is
 useful to recall that $ | y|_{-1}  $ and $ | (\omega I - \A)^{-1}
y|_{} $ are equivalent norms on $ \DASP $, for every $
\omega\in\rho(A)   = \rho (\A)=\overline{  \rho(A^*)}$ (here $\rho $
 indicates
 the resolvent set and the overbar denotes the complex conjugate).
 In other words,
 $ \DASP $ is the completion of $H$ with respect to the norm
  $ | (\omega I - A)^{-1}
 \cdot |_{}$, for any $\omega \in \rho(A)$.

Let $ U $ be an Hilbert space. A ``control'' is an element of $ L^2_{\rm loc}(0,+\infty;U) $. Let
$ B\in{\mathcal L}(U,({\rm dom}\, A^*)') $ and let us consider the
control process on $ \DASP $ described by
\begin{eqnarray} \label{base}
\dot y = \A y + Bu, \qquad y(0) = y_0 \in H\,.
\end{eqnarray}
This equation makes sense in $ \DASP $, for every $
y_0\in\DASP $, but we only consider  initial conditions $ y_0\in H
$. It is known that the transformation
\begin{equation}
\ZLA{eq:TrasfCONTROevolNELduale}
u(\cdot)\ \longrightarrow\  ( L u)(t)
\qquad \text{where} \;\;\;
 ( L u)(t) :=
\int_0^t e^{\A(t-s)}Bu(s)\ZD s
\end{equation}
 is continuous from $ L^2(0,T;U) $ into $ C([0,T];\DASP) $,
 for every $ T>0 $.
The class of systems we study is
 identified by pairs $ (A,B) $ with the following property:
\begin{Hypothesis}\ZLA{Hyp:0} We
have $ B\in{\mathcal L}(U,({\rm dom}\, A^*)') $ and,
 for every $ T>0 $, the
 transformation~(\ref{eq:TrasfCONTROevolNELduale})
is linear and continuous from $ L^2(0,T;U) $ into $ L^2(0,T;H) $.
\end{Hypothesis}
Clearly,  the case of distributed controls, i.e., $B \in {\mathcal
L}(U,H)$, fits Hypothesis~\ref{Hyp:0} (in such case,
 the
 transformation~(\ref{eq:TrasfCONTROevolNELduale})
is linear and continuous from $ L^2(0,T;U) $ into $C([0,T];H) )$.
Examples of boundary control systems which satisfy our condition are
in Section~\ref{sec:ClassesSYSTEMS}.

>From now on we consider
 $\omega \in \rho (A)$, \emph{which is fixed once and for all,}
and  introduce the operator
\begin{align} \label{dd}
D= ( \omega I-\A)^{-1}B\in {\mathcal L}(U,H)\,.
 \end{align}
 By definition, the solution of system~(\ref{base}) is
\begin{equation}
\label{eq:SoluBASE2}
 y^{y_0,u}(t)=e^{At}y_0 + \int_0^t e^{\A(t-s)} Bu(s)\ZD s=
 e^{At}y_0 + \int_0^t e^{\A(t-s)}(  \omega I-\A) D u(s)\ZD s\,.
\end{equation}
This is a continuous $ \DASP $-valued function and belongs to $
L^2_{\rm loc}(0,+\infty;H) $ thanks to  Hypothesis ~\ref{Hyp:0}.
Integration by parts shows that:
\begin{Lemma}\ZLA{Lemma:ContY}
If $ u\in C^1([0,+\infty);U) $
 then  $ y^{y_0,u} $ belongs to $ C([0,+\infty);H) $.
\end{Lemma}

Now we give the definitions of null controllability and NCVE,
adapted to our system, by taking into account the fact that if $
u\in L^2 _{{\rm loc}}(0,+\infty;U) $
 then the  integrals
 in~(\ref{eq:SoluBASE2}) belong to   $ L^2 _{{\rm loc}}(0,+\infty ; H) $,
 and point-wise evaluation in $ H $ in general is meaningless.

\begin{Definition}\ZLA{defi:CONTROREST}
 We say that $ y_0\in H $ can be steered to the rest in time (at
most) $ T $ if there exists a control $ u\in L^2 _{{\rm
loc}}(0,+\infty;U) $ whose support is contained in $ [0,T] $ and
such that the support of the corresponding
solution~(\ref{eq:SoluBASE2})
 is contained
 in $ [0,T] $ too.

  System~(\ref{base}) is null controllable
 if   every $ y_0\in H $
  can be steered to the rest in a  suitable
   time  $ T_{y_0} $ at most.

 System~(\ref{base}) is null controllable in time (at most) $ T $
 if every $
y_0\in H $ can be
 steered to the rest in time at most $ T $.
\end{Definition}
 In connection
with this definition see also Lemma~\ref{Lemma:uniFOcontrolTIME}.

 Controllability in time $ T $ implies controllability at every
larger time.  Note that
 if $ u $ steers $ y_0 $ to the rest in time at most $
T $, then we have
\[
\int_0^T e^{\A(t-s)}Bu(s)\ZD s=-e^{At}y_0,\qquad {\rm a.e.} \quad
t>T,
 \]
 and so the integral
is represented by a continuous function for $ t>T $.

 The control $ u $ which steers $ y_0 $ to zero in time $ T $ needs not be unique.
 Then, we define:

 \begin{Definition}\ZLA{defi:NCVE}
 Let $ y_0 \in H$ be
an element which can be steered to the rest. We say that this
element is NCVE if  for every $\ZEP>0  $ there exists a control $
u_\ZEP $ such that
\begin{itemize}
\item it steers $ y_0 $ to the rest in time
at most $ T_\ZEP $  (i.e., the control time $ T $
  depends on $ \ZEP $,   $ T=T_\ZEP $,
  and the support of $ u $ is in $ [0,T_{\ZEP}] $);
\item the $ L^2(0,+\infty;U) $ norm of $ u $ is less then $ \ZEP$:
\[
\int_0^{+\infty} | u(s)|^2\ZD s=\int_0^{T_\ZEP} |u(s)|^2\ZD s\leq
\ZEP^2\,.
 \]
\end{itemize}

If every element of $ H $ is NCVE, then we say that
system~(\ref{base}) is NCVE.
 \end{Definition}

 As a variant to
  Definitions~\ref{defi:CONTROREST}
  and~\ref{defi:NCVE}, we introduce also:
 \begin{Definition}
 Let $ \mathcal D $ be a subspace of $ H $.
 If every initial condition $ y_0\in {\mathcal D}  $ can be steered to the rest in time $ T $ then we say that the system is null controllable on $ {\mathcal D} $ in time $ T $ (note that we don't require that the trajectory which joins $ y_0 $ to zero remains in the set $ \mathcal D $).

We say that the system is NCVE on $ {\mathcal D} $
   if
for every $ y_0\in {\mathcal D} $ and every $\ZEP>0  $ there exists a
control $ u_\ZEP $ such that
\begin{itemize}
\item it steers $ y_0 $ to the rest in time $ T_{\epsilon} $
 (i.e., the control time $ T $
  depends on $ \ZEP $, i.e., $ T=T_\ZEP $);
\item the $ L^2(0,+\infty;U) $ norm of $ u $ is less then $ \ZEP$:
\[
\int_0^{+\infty} | u(s)|^2\ZD s=\int_0^{T_\ZEP} |u(s)|^2\ZD s\leq
\ZEP^2\,.
 \]
\end{itemize}
 \end{Definition}

\subsection
{\ZLA{sec:ClassesSYSTEMS}Classes of systems which fit our framework}

Essentially, controllability    has been studied for ``parabolic''
and ``hyperbolic''  type  systems.

\smallskip \noindent (i)
 {\it Parabolic systems\/} can be described, in a unified way,
as follows.

 The operator $ A $ generates a holomorphic semigroup and,
following ~\cite[Section 0.4 and Chapter 1]{LT1},
 there exists $\omega \in
\rho(A) = \rho (\A)$ and $\gamma \in [0,1)$
   such that
\begin{align} \label{dom}
B \in {\mathcal L}\left  (U , \big(\text {dom}\, (\omega -A^*)^{\gamma} \big)'\right ).
\end{align}
 Note that \eqref{dom} implies the estimate
\begin{eqnarray} \label{DiseSINGOLARE}
\|e^{ \A t} B\|_{{\mathcal L}(U, H)} \leq \frac {Me^{\omega_1 t}}
{t^{\gamma}}, \;\;\; t
>0.
\end{eqnarray}
for some  $M>0$, $\omega_1 \in \R$ (see~\cite[Section 0.3]{LT1},
\cite[Chapter 3]{BD1} and~\cite{TW}; see also
  Appendix~\ref{AppBASICsetting}); recall that
$\big(\text {dom}\, (\omega -A^*)^{\gamma} \big)' \subset \DASP$ with
continuous and dense injection).

Using \eqref{DiseSINGOLARE}, one can show that {\it  Hypothesis~\ref{Hyp:0} holds in this case.\/} Indeed, the integral
in~(\ref{eq:SoluBASE2})
 does not converge in the space $ H $ for every $ t $
 but,
 using the Young inequality for convolutions,
  it defines an $H$-valued
 locally square  integrable function for every locally square
integrable input $ u $.  Formula~(\ref{eq:SoluBASE2})
 defines  the unique
 solution of eq.~(\ref{base})  with values in $H$,
 which however does not have a pointwise sense in general.

A recent example of parabolic system will be  considered in Section
3.

The singular inequality~(\ref{DiseSINGOLARE}) holds for certain
important classes of interconnected systems, as studied for example
in~\cite{BUCCI,L}, even if they do not generate holomorphic
semigroups.

\smallskip \noindent (ii)
{\it Hyperbolic systems \/}
are  further important examples of systems which fit our framework,
see ~\cite{LT2,Kom} and~\cite[p.~122]{TW}.  In spite
of the fact that this class lacks of a plain unification,   it turns
out that in this case the following important property, first proved
for the wave equation with Dirichlet boundary control
in~\cite{LASTRiDIRIC1,LASTRiDIRIC}, holds: the function $ y(t) $ is
even continuous in time.

\smallskip We listed  earlier  systems which fit
our Hypothesis~\ref{Hyp:0}. However,  null controllability cannot
be studied ``in abstract'': it has to be studied separately in
concrete cases and these are too many to be cited here. So, we
confine ourselves to note that controllability for several
hyperbolic type problems is studied
 in~\cite{AVDivan,LasieTriggContro,LIONS,Kom,KomoLORE,Soriano}; controllability
for parabolic type equations is studied
in~\cite{CARA,MIl,TENENBAUM-TUCSNAK,XUZhang} and references therein.
 Note that controllability for heat-type equations is often achieved
using smooth controls, so that the resulting trajectory $ y(t) $ is
even continuous.

An
overview  on controllability both of hyperbolic and parabolic type
equations is~\cite{LIzhang,ZUAZUA}.

\subsection{Key results and discussion}

As we  have already said, our point of departure is
 paper~\cite{PZ} which proves the following result,
in the case of distributed controls, i.e., the case that
 $ {\rm im}\,
D\subseteq {\rm dom}\, A $ (recall that $D$ is defined in
\eqref{dd}) so that $ B\in {\mathcal L}(U,H) $: {\em under suitable
assumptions on the spectral properties of the operator $ A $, NCVE
is equivalent to null controllability at some time $ T>0 $. \/} An
interesting interpretation of this result is that  for this class
of systems\/ {\em  NCVE does not
depend on the control  operator provided that this operator is so
chosen to guarantee null controllability at a certain time $ T $.\/}

 Now we state our main results, which we split
 in Theorems~\ref{primo}
 and~\ref{secondo}. We don't try to unify them, since they are proved
using different ideas but, in the most important cases for the
applications, they can be combined to get a necessary and sufficient
condition for NCVE, see Corollary~\ref{coro:rema:nounif}.

  We recall that
 a {\em  reducing subspace\/}
 $ E $
 for a $ C_0 $-semigroup $e^{A t}$ on $ H $ is a closed subspace
  of $H$ such that both $ E $
  and one of its complementary subspaces are invariant for the semigroup:
 \[
 e^{At}x\in E, \qquad \forall x\in E\,,\ \forall t\geq 0
  \]
  and the same for one complement of $ E $.

It is possible to prove   that the restriction of $ e^{A t} $ is a $
C_0 $-semigroup on $ E $ and that $ A\left ( E\cap \left ({\rm
dom}\, A\right )  \right ) \subseteq E $  (the restriction of $ A $
to $ E $ is the infinitesimal generator of $ e^{At}  $ on $ E $).

 The
necessary condition for NCVE is given by the next theorem:

\begin{Theorem} \label{primo}
 Assume Hypothesis~\ref{Hyp:0} and  suppose
  the existence of a reducing subspace $ E $ for
$e^{At}$, such that  $e^{-At}$ generates a $ C_0 $-\emph{group} on $
E $ which  is \emph{exponentially stable    }    (for $ t\to+\infty
$). Then, the system \eqref{base} is not NCVE.
\end{Theorem}

A consequence is:
\begin{Corollary} Assume Hypothesis~\ref{Hyp:0},
If    $ \sigma(A) $ has an isolated point with \emph{positive} real
part, then the system \eqref{base} is {\em not\/} NCVE.
\end{Corollary}

In fact,~\cite{PZ}
proves the existence of the subspace $ E $ in Theorem~\ref{primo},
under the assumption of the corollary.

 Now we come to the second theorem.
We recall that $x \in H$
 is a generalized eigenvector of $A$ associated to the
 eigenvalue $\lambda \in \C$ if
 $x \in \bigcup_{k \ge 1}$Ker$[(\lambda I - A)^k]$ and we recall the standard notation   for the spectral bound
\[
s(A)=\sup\{\zreal\lambda\,,\quad \lambda\in\sigma(A)\}.\,
 \]
where $s(A) = - \infty$ if $\sigma(A)$ is empty.

 Now we introduce the
following assumption, which slightly generalizes the one
in~\cite[(ii)  Hypothesis 1.1]{PZ}.
\begin{Hypothesis} \label{serve}
There exist
closed linear subspaces $H_s$, $H_1$ of $H$ such that:
\begin{itemize}
\item $H= H_s \oplus H_1$;
\item for every $ x\in H_s $
we have
\[
\lim _{t\to+\infty} e^{At} x=0\,;
 \]
\item the subspace $ H_1 $
is invariant for the semigroup and
the set of all the generalized eigenvectors of $A$ contained in $H_1$ is
linearly dense in $H_1$.
\end{itemize}
In the definition the subspace $H_1$ can be $\{0\}$.
If $\sigma(A)  = \emptyset$ we set $H_1 =\{0\}$.
\end{Hypothesis}

\noindent We note that the assumption in~\cite{PZ} is slightly
stronger in that~\cite{PZ} assumes that $ H_s $ is an invariant
subspace for the semigroup, and that the semigroup restricted to $
H_s  $ is exponentially stable.

  We have:

\begin{Theorem} \label{secondo}
 Assume  Hypotheses~\ref{Hyp:0} and~\ref{serve}
 and furthermore suppose that $ s(A)\leq 0 $.
  If   system~(\ref{base}) is
  null controllable at some time $T>0$,
then     it is NCVE.
\end{Theorem}

The ideas used in the proof of both Theorems~\ref{primo}
and~\ref{secondo} are different from those used in the proofs of the
corresponding results in~\cite{PZ}. In particular, the proof of
Theorem~\ref{secondo} relies on
  the
{\it Yakubovich theory of the regulator problem   with stability},
    and the corresponding
{\it Linear Operator Inequality,} that can be found
 in~\cite{LW,PandLOI1,PandLOI2,PandLOI3}.

 Clearly,  the spectral condition in Hypothesis~\ref{serve} is
satisfied by most of the systems encountered in practice,  when the
``dominant part'' of the spectrum is a sequence of eigenvalues (in
particular, if  $ A $  has compact resolvent). Hence, for all these
systems, Theorems~\ref{primo} and~\ref{secondo} can be combined to
get a \emph{necessary and sufficient} condition for NCVE which
depends only on the spectrum of $ A $, provided that  null
controllability holds. For example we can state the following
Corollary. Recall that a $C_0$-semigroup $e^{At}$ is called
\emph{eventually compact} if there exists $t_0>0$ such that
 $e^{At}$ is a compact operator for any $t \ge t_0$; moreover any
 differentiable semigroup such that
  its generator has compact resolvent is in particular
  an eventually compact semigroup, see
  \cite[Theorem 3.3, page 48]{Pa}.

\begin{Corollary}\ZLA{coro:rema:nounif}
 Assume Hypothesis~\ref{Hyp:0} and suppose that the semigroup is
  eventually compact.
  If $s(A) \le 0$ then null controllability and NCVE  are
 equivalent properties. When $ {s(A)>0}$
 the system is not NCVE.
\end{Corollary}
\begin{proof}

The spectrum of $ A $ is a sequence of eigenvalues (this is well
known when $ (\ZOM I-A)^{-1} $ is compact, and it is true also if
the semigroup is eventually compact, see~\cite[p.~330]{Engel};
 note
 that  the spectrum might be empty in this case). Furthermore, under
 the stated assumptions we have  (see~\cite[p.~330]{Engel}),
  for any $r \in \R$, the set
 \begin{align*} \label{comp} \{ \mu \in \sigma(A)\, : \,
\mbox {\rm Re} (\mu) \ge r \} \;\; \mbox { is finite or empty.}
\end{align*}

As we noted, when the semigroup is eventually compact,
 the spectrum of $ A $ might be empty.   In this case we can choose
  $ H_1=0 $  and $ H_s=H $ since the semigroup is
exponentially stable on $ H $, see~\cite[p.~250-252]{Engel}. So, if
 null controllability holds we have also NCVE.

 Let the spectrum be not empty.
 If   $ {s(A)>0}$ then there exists
 an eigenvalue $\lambda$ with positive real part: one can easily
 show (see for example~\cite[Sect.~2.1]{PZ}) that the
 subspace $E$ of all generalized eigenvectors associated to $\lambda$
  is reducing for $e^{At}$ and that $e^{-At}$ generates a group on
  $E$ which is exponentially stable. We are
  in the case of Theorem~\ref{primo} and NCVE does not hold.

Let now $ s(A)\leq 0 $ and take
 $ c< s(A) $.  Let $ H_1 $ be the invariant subspace of $ H $
spanned by all the generalized eigenvectors associated to the (finitely many)
eigenvalues   with real part larger then  $ c $. Let $ P_{H_1} $ be the corresponding spectral projection     and set  $
H_s=\left (I-  P_{H_1}\right ) H  $. Then, from~\cite[p.~267]{CZ},
the semigroup is even
 exponentially stable on $ H_s $
and the conditions of Theorem~\ref{secondo} are satisfied, hence
NCVE holds.
\end{proof}

We conclude this introduction with the following
 observation which extends
 a property of null controllable systems
 proved by many people for
 distributed
 controls (see~\cite{fuhrman,Rolewicz,vanNeerve})
and likely known at least for some boundary control systems, in spite of the fact
that we cannot give a precise reference:
\begin{Lemma}\ZLA{Lemma:uniFOcontrolTIME} Assume Hypothesis~\ref{Hyp:0} and suppose that  every $ y\in H $ can be steered to
rest in a time $ T_y $. Then:
\begin{itemize}
\item there exists a time $ T_0 $ such that system~(\ref{base}) can be steered to the rest in time $ T_0 $;
\item there is a ball $ B(0,r) $ (centered at $ 0 $, radius $ r>0 $) and a number $ N $ such that every element of $ B(0,r) $ can be steered to the rest using a control whose $ L^2 $-norm is less then $ N $.
\end{itemize}

\end{Lemma}
\begin{proof} The proof is the same as for distributed systems: we introduce the sets $ E_{T,N} $ of those elements
$ y\in H $ which can be steered to the rest in time (at most) $ T $ and using controls of norm at most $ N $. These sets are closed, convex and balanced. Furthermore, they grow both with $ T $ and with $ N $.

Every $ y $ belongs to a suitable $ E_{T,N}  $ so that
\[
H=\cup E _{N,N}\,.
 \]
 Baire Theorem implies the existence of $ N_0 $ such that $ E_{N_0,N_0}  $ has interior points.

 The set $ E_{N_0,N_0}  $ being convex and balanced, $ 0 $
  is an interior point,
 i.e., any point of a ball centered at zero can be steered to the rest in time $ T=N_0 $ and the $ L^2 $-norm of the corresponding control is less then $ N _0$. This is the second statement and it implies that every $ y\in H $ can be steered to the rest in time $ T=N_0 $.
 \end{proof}

 In conclusion, we see that
null controllability   and null controllability at a fixed time $
T>0 $ are equivalent concepts.

\section{Proof of the main results}

 First we state two
lemmas which have an independent interest.

  Let $  y(t) $
  solve equation~(\ref{base}). Then,
  $  x(t)=(  \omega I- \A)^{-1}y(t) $ solves the equation
  \begin{equation}
\ZLA{eq:baseREGOL} \dot x=Ax + Du\,,\qquad
x(0)=x_0=( \omega I-A)^{-1}y_0\in {\rm dom }\,A\,.
\end{equation}
Consequently, every control which steers $ y_0 $ to zero, steers
also $ x_0 $ to zero, and conversely. Therefore,  we have:
\begin{Lemma}\ZLA{Lemma:base}
Assume Hypothesis~\ref{Hyp:0}. There exists $ T>0 $ such that
system~(\ref{base}) is null controllable in time $ T $ if and only
if system~(\ref{eq:baseREGOL}) is null controllable on $ {\mathcal
D}={\rm dom}\,A $ in the same time $ T $;
 system \eqref{base}  is NCVE if and only if
system~(\ref{eq:baseREGOL}) is NCVE on $ {\mathcal  D}={\rm dom}\,A $.
\end{Lemma}

>From now on, $ \mathcal D $ will always denote $ {\rm dom}\, A $, i.e.,
\[
 {\rm \mathcal  D}={\rm dom}\,A \,.
 \]
 The second preliminary result is the following lemma:

 \begin{Lemma}\ZLA{Lemma:stimaCONTR(OLLO}
Assume Hypothesis~\ref{Hyp:0} and suppose that  system~(\ref{base})
is null controllable   in time $ T $. Then there exists a number
$ M>0
$ such that for every $ y_0\in H $ there exists a control $
u^{y_0,T}(t) $ which steers $ y_0 $ to $ 0 $ in time $ T $ and such
that:
\[
\int_0^{T} |  u^{y_0,T}(t)|^2\ZD t\leq M |y_0|^2\,.
 \]
\end{Lemma}

\begin{proof}
We already noted that we can equivalently control system~(\ref{eq:baseREGOL}) to zero on
$\mathcal D$,  i.e., we can solve
\begin{equation}\ZLA{eq:ControNULLtempoT}
-e^{AT}(\omega I -A)^{-1}y_0=\int_0^T e^{A(T-s)} Du(s)\ZD s
\end{equation}
and by assumption this equation is solvable for every $ y_0\in H $.
We introduce the operator $ \Lambda_T :  L^2 (0,T; U) \to H$  as
\begin{equation}
\label{eq:FirstDefinitionofLAMBDA} \Lambda_T u=\int_0^T e^{\A(T-s)}
Du(s)\ZD s\,.
\end{equation}
So, null controllability at time $ T $ is equivalent to
 \[
 {\rm im}\, e^{AT} (\omega I -A)^{-1}\subseteq {\rm im} \Lambda_T.
  \]
  The operator $ \Lambda_T $ is continuous,
  \[
  \Lambda_T \in {\mathcal L}\left (
  L^2(0,T;U),H
  \right ).
   \]
   Let us
introduce the continuous operator $ Q_T=\Lambda_T\Lambda_T^*\,. $
Its kernel is closed and its restriction to the orthogonal of the
kernel is invertible with closed inverse. Let us denote   $
Q_T^\dagger $ this inverse, so that the  control   which steers  $ ( \omega
I-A)^{-1}y_0 $  to zero in time $ T $  and which has minimal $
L^2(0,T) $ norm
  is
\begin{equation}
\ZLA{eq:FormACONTRnormaminima}
   u^{ y_0,T}(t)=-\Lambda_T^*Q_T^\dagger e^{AT}( \omega I-A)^{-1}y_0
 =-D^* e^{A^*(T-t)}Q_T^\dagger e^{AT}( \omega I-A)^{-1}y_0\,.
\end{equation}
  The closed operator $ Q_T^{\dagger} e^{AT}( \omega I-A)^{-1} $ being everywhere defined, it is continuous, so that
  \[
  \|  u^{y_0, T}\|_{L^2(0,T;U)}\leq M |y_0|\,,\qquad M=M_T\,,
   \]
   as wanted.
   \end{proof}

\begin{Remark}\label{Rema:SuiDIVERSItempiniz}{
  We note:
\begin{itemize}
\item
The function  $   u ^{y_0, T} (t)$, extended with $ 0 $ for $ t>T
$, produces a solution $ y(t) $ to Eq.~(\ref{base}), which has
support in $ [0,T] $.

\item  We can work with any initial
time $ \tau $ instead of   the initial time $ 0 $. If the system is
null controllable in time at most $ T $, then any ``initial
condition'' assigned at time $ \tau $ can be steered to rest on a
time interval still of duration $ T $, i.e., at the time $ T+\tau $
and the previous Lemma~\ref{Lemma:stimaCONTR(OLLO} still holds, with
the constant $ M $ depending solely on the length of the
controllability time, i.e., the same constant $ M_T $ can be used
for every initial time $ \tau $.

\end{itemize}
} \end{Remark}

\subsection{Proof of Theorem~\ref{primo},  {\it i.e., if NCVE holds then
 the  subspace $ E $ does not exist.\/}}

The proof in~\cite{PZ} relays on a precise study of the quadratic
regulator problem and the associated Riccati equation. Here we
follow a different route: we prove that the existence of the
subspace $ E $
 implies that system~(\ref{base}) is not NCVE.

 Let $ E^C $ be the complementary subspace of $ E $ which is
invariant for the semigroup.
 Let $ y_0\neq 0 $ be any  point of $ E $.
If it cannot be steered to $ 0 $ then system~(\ref{base}) is not
null controllable, hence even not NCVE. So, suppose that there
exists a control $ u $ which steers $ y_0 $ to zero in time $ T $.
Then we have, for every $ t>T $,
\[
e^{At}y_0=-\int_0^t e^{\A(t-s)}Bu(s)\ZD s,
 \]
i.e.,
 \[
e^{AT}( \omega I-A)^{-1} y_0=-\int_0^T e^{A(T-s)}Du(s)\ZD s\,.
\]
Let now $ P_E $ be the projection of $ H $ onto $ E $ along  $ E^C
$.   We have
\begin{align} \label{referee}
e^{AT}( \omega I-A)^{-1} y_0=-\int_0^T e^{A(T-s)}P_E Du(s)\ZD s
-\int_0^T e^{A(T-s)}(I-P_E )Du(s)\ZD s\,.
 \end{align}
 The left hand side belongs to $ E $ so that the last integral is
zero since  $ (I-P_E) $ commutes with the semigroup
  due to the fact that $ E $ is a reducing subspace.
  Then,
\[
e^{AT}( \omega I-A)^{-1} y_0=-\int_0^T e^{A(T-s)}P_E Du(s)\ZD s\,.
 \]
Hence we have also
\[
 ( \omega I-A)^{-1} y_0=-\int_0^T e^{-A  s }P_EDu(s)\ZD s
 \]
 {\em since $ A $ generates a group on $ E $.\/}   Note
that this equality in particular implies that $ P_ED\neq 0 $ since
the left hand side is not zero.

 We assumed that $ e^{-At} $ is exponentially stable on $ E $, i.e.,
  we assumed the existence of $ M>1 $ and $ \gamma>0 $ such that

 \[
 \left | e^{-A s}y \right |\leq M e^{-\gamma s} | y | \qquad
     \mbox{for all $ s>0$ and for all  $ y \in E $}\,.
  \]
  So, using Schwarz inequality we see that:
 \[
 \left | ( \omega I-A)^{-1} y_0\right |\leq
\frac{M\| P_ED\|_{{\mathcal L}(U, H)}}{\sqrt{2\gamma}}\| u\|
_{L^2(0,T; U)}\,.
  \]
  This is an
estimate \emph{from below} for the $ L^2(0,T) $-norm of any control
which steers $ y_0 $ to the rest, and this estimate does not depend
on $ T $. Hence, the system is not NCVE, as we wished to prove.\qed

\subsection{Proof of Theorem~\ref{secondo}, {\it i.e.,
 null controllability  and $ s(A)\leq 0 $ implies  NCVE\/} }

We introduce a new notation.  Since we need to consider solutions of
equation~(\ref{base}) with initial time $ \tau $, possibly different
from $ 0 $, we introduce
\[
y(t;\tau,y_0,u)
 \]
 to denote the solution of the problem
 \[
 y'={\A}y+Bu \qquad t>\tau\,,\qquad y(\tau)=y_0\,.
  \]
  Furthermore,
  when $ \tau=0 $, we shall write $ y(t;y_0,u) $
  instead of $ y(t;0,y_0,u) $.
  Comparing with \eqref{eq:SoluBASE2}, we have
  \[
  y(t;y_0,u)=y^{y_0,u}(t)\,.
   \]
We first   give  a different formulation
 of the problem under study. To this purpose
 we introduce the following functionals
 $ I(y_0) $ and $ Z(y_0) $:

\begin{eqnarray}
\ZLA{eq:FormaQUADR}   I(y_0)=\inf_{u\in{\mathcal U}(y_0)} {J(y_0;u)}
\,,\qquad
J(y_0; u)=\int _{0}^{+\infty}|u(s)|^2\ZD s\\
\nonumber   {\mathcal U}(y_0) = \left \{ u\in L^2(0,+\infty; U)\,:\
 y(\cdot; y_0,u)\in L^2(0,+\infty;H)  \right
\}
\end{eqnarray}
and
\begin{equation}
 \label{eq:costoZ}
 Z(y_0)
 = \inf_{t > T}  \; \left ( \inf   \Big \{ \int_0^{t}
  | u(s) | ^2  ds\Big\}\right ),
\end{equation}
where, for each $ t>T $, the infimum in braces is computed on
those controls $ u  $ which steers $ y_0 $ to the rest in time at
most $ t $ (i.e., the supports of  $u $ and $ y^{y_0,u} $ have to be
contained in $ [0,t] )$ (recall Lemma~\ref{Lemma:uniFOcontrolTIME}).

Using
Lemma~\ref{Lemma:stimaCONTR(OLLO}, we can prove:

\begin{Theorem} \label{starAt} Assume Hypothesis~\ref{Hyp:0}.
 Let system~(\ref{base})
 be null controllable in time $ T $. Then we have
 \[
 I(y_0)=Z(y_0)\,.
  \]
\end{Theorem}
 \begin{proof}
  It is clear that
 $ I(y_0)\leq Z(y_0) $ for every $ y_0\in H $. We prove the converse
inequality.

 Let us fix any $ y_0\in H $.
Null controllability  implies that $ I(y_0)<+\infty $ for every
$ y_0 $ so that for every $ \ZEP>0 $ there exist a control $ u_{\ZEP} \in {\mathcal U}(y_0)
 $   such that for  every $ S>0 $  we have
 \begin{equation}\ZLA{DisePERIequalZ}
 \int_0^S  |u _{\ZEP}(t) |^2\ZD t< I(y_0)+\ZEP\,.
  \end{equation}
 The condition
$ u _{\ZEP}\in {\mathcal U}(y_0) $ implies  $ y^{y_0,u_\ZEP}\in
L^2(0,+\infty; H) $ and so for  every $ \ZSI>0 $ we have  $ \left |
y^{y_0,u_\ZEP}\right |^2 _{L^2(R,R+1;H)}<\ZSI$ for $ R $
sufficiently large.

   There exists a
sequence $\{ u_n\}$ in $C^1 ([0,R+1];U) $, which converges to $
u_\ZEP $ in $ L^2(0,R+1;U) $. So, there exists an index $ N$ such
that inequality (\ref{DisePERIequalZ})   holds for $ u_N $ and
furthermore $ \left| y^{y_0,u_N}\right|^2 _{L^2(R,R+1; H)}<\ZSI$.
Hence we can find $ S_\ZSI\in [R,R+1] $ such that
  the continuous function $ y^{y_0,u_N}  (t) $ satisfies
 \[
 | y^{y_0,u_N}  (S_\sigma )|^2=|y(S_\sigma,y_0,u_N)|^2<\ZSI\,.
  \]
Null controllability holds also
 on $ [S _{\ZSI}, S_\ZSI+T] $
and Lemma~\ref{Lemma:stimaCONTR(OLLO} can be applied on this
interval (see also  Remark~\ref{Rema:SuiDIVERSItempiniz}). Hence,
there exists a control $ \tilde u $ with support in   $ [S _{\ZSI},
S_\ZSI+T] $ which steers to the rest in time $ T $ the ``initial
condition'' $ y^{y_0,u_N}  (S_\sigma ) = y(S_\sigma,y_0,u_N)$,
assigned at the ``initial time'' $ S _{\ZSI} $.
Lemma~\ref{Lemma:stimaCONTR(OLLO} shows
 that the square norm of this control is less then $ M\ZSI $.

 Now we apply first
the control  $ u_N $, on $ [0,S_\ZSI] $, and after that the control
$ \tilde u $. In this way we  steer  $ y_0 $ to the rest in time at
most $ S_\ZSI+T $ and the square of the $ L^2 $ norm of the control
is less then $ I(y_0)+\ZEP+M \ZSI $. This shows that
 \[
 Z(y_0)\leq I(y_0)+\ZEP+M \ZSI\,.
  \]
  The required
inequality follows since $ \ZEP>0 $ and $ \ZSI>0 $ are arbitrary.
\end{proof}

This theorem shows:
\begin{Corollary}\label{Corollary:IyENCVE}
Assume Hypothesis~\ref{Hyp:0} and suppose that
  system~(\ref{base})
 is null controllable in time $ T $. Then System~(\ref{base})
is NCVE if and only if $ I(y_0)=0 $, for every $ y_0\in H $.
\end{Corollary}

So, our goal now is the proof that, under the assumptions of
Theorem~\ref{secondo}, we have $ I(y_0)=0 $.

The study of the value function $ I(y_0) $ defined above is the
object of the so-called   theory of the \emph{quadratic regulator
problem with stability} or \emph{Kalman-Yakubovich-Popov Theory}. It
has been studied, for special classes of distributed control
systems, in~\cite{LW,PandFREQdomREG,PandLOI1,PandLOI2,PandLOI3}.
But, we need an improved version of the    results of this theory,
i.e., we need the following theorem:
\begin{Theorem}\ZLA{teo:KYP}
 Assume Hypothesis~\ref{Hyp:0}
and suppose that  system~(\ref{base}) is null controllable.   Then there exists $ P=P^*\in{\mathcal L}(H) $ such that for every $ y_0\in H $ we have
\begin{equation} \label{eq:DeFInixDiP}
I(y_0)=\langle y_0,Py_0\rangle\,.
\end{equation}
Furthermore, the operator $ P $ satisfies the following inequality, for any   $y_0 \in H$, $u \in L^2_{loc}(0, + \infty; U)$,
 \begin{equation}
\ZLA{eq:INteINEQINTEGRALform}
   \ZL  Py(t;y_0,u),y(t;y_0,u)\ZR-\ZL  Py_0,y_0\ZR+
   \int_0^t |u(s)|^2\ZD s\geq 0 \quad  \text{a.e.}\;\; t \ge 0\,.
    \end{equation}
\end{Theorem}
Inequality~(\ref{eq:INteINEQINTEGRALform})
 is
called {\em Linear Integral Inequality\/}---shortly (LOI)---or {\em
Dissipation inequality\/}---shortly (DI)---in integral form.

\medskip
  The proof of Theorem ~\ref{teo:KYP}
requires some preliminary
   lemmas.

 \smallskip 
  We first note that the functional $I$ defined in
 \eqref{eq:FormaQUADR} verifies   $I (\lambda y_0) = |\lambda|^2
I(y_0)$, $\lambda \in \C$, $y_0 \in H$. An
 obvious consequence is
\begin{equation}
\label{eq:ProBASEiX}
\left\{\begin{array}{lll}
I(x)=I(-x) &{\rm and} &I(0)=0\,,\\
|\alpha|=|\beta|&{\rm implies}& I(\alpha x)=I(\beta x), \;\;\;
 \alpha, \beta \in \C, \; x \in H.
\end{array}\right.
\end{equation}
  Then   we give a representation of $ I(y_0) $.
\begin{Lemma}\ZLA{teo:defiFOqquadr} Assume Hypothesis~\ref{Hyp:0} and suppose  that
  system~(\ref{base})
 is null controllable in time $ T $. Then
 there exists an operator $ P $
defined on $H$ such that
\begin{equation}\ZLA{eq:ladefiPPRIMAvers}
I(y_0)=\ZL   y_0,Py_0\ZR\,.
 \end{equation}
 The operator $ P $ has the following properties:
 \begin{itemize}
\item[\bf a)]  $\ZL Px,\xi\ZR=\ZL x,P\xi\ZR\qquad \forall x\,,\ \xi \in H$;
\item[\bf b)]  $P(x+\xi)=Px+P\xi\qquad \forall x\,,\ \xi \in H$;
\item[\bf c)]  the equality $P(qx)=qPx$ holds for all $ x\in H $ and every
complex  number $ q $ with \emph{rational} real and imaginary parts.
\item[\bf d)] $ \ZL y_0,P y_0\ZR\geq 0 $ for all $ y_0 \in H$.
\end{itemize}
\end{Lemma}
\begin{proof}
The proof uses~\cite[Sect.~9.2]{greub} (adapted to complex Hilbert
spaces) and it is an adaptation of the proof of~\cite[Theorem
5]{FavPand}.

 Recall  that    ${\mathcal U}(y_0)
$ is not empty  since  system~(\ref{base})  is null
controllable. So,
null controllability implies   that $ I(y_0) $ is finite for
every $ y_0\in H $.

Let us fix $ x_0 $ and $ \xi_0 $ in $ \mathcal D = {\rm dom}\, A$ and controls
$ u\in {\mathcal U}(x_0) $ and $ v \in {\mathcal U}(\xi_0)$ Then we have
\[
y(t;x_0\pm \xi_0,u\pm v)=y(t;x_0,u)\pm y(t;\xi_0,v)
 \]
  and $ J $ satisfies the parallelogram identity
  \[
  J(x_0 + \xi_0; u+v)+J(x_0 - \xi_0; u-v)=
  2\left [ J( x_0;u)+J(\xi_0;v)\right ]\,.
   \]
We must prove that $ I(x) $ satisfies the parallelogram identity
too. This part of the proof is the same as that in~\cite[Theorem~5]{FavPand} and it is reported for completeness.

We fix $x$ and $\xi$ and $\ZEP>0$ and we choose $u_x$ and $u_\xi$,
corresponding to the initial conditions $x$ and $\xi$, such that
$$
J(x;u_x)<I(x)+\ZEP/2\,,\qquad J(\xi; u_\xi)<I(\xi)+\ZEP/2.
$$
Then
\begin{eqnarray*}
&& J(x + \xi;u_x+u_\xi )+ J(x - \xi;u_x-u_\xi ) \\ &&
=2J(x;u_x)+2J(\xi; u_\xi )<2\left[ I(x)+I(\xi ) \right]+2\ZEP .
\end{eqnarray*}
This proves the inequality
\begin{align} \label{paral}
I(x+\xi )+I(x-\xi )\leq 2\left[ I(x)+I(\xi )\right]\,.
\end{align}
We prove that the inequality cannot be strict; i.e., we prove that
if $\ZEP$ satisfy
\begin{equation}
\ZLA{eq:PerLidPARALL}
I(x+\xi )+I(x-\xi )\leq 2\left[ I(x)+I(\xi )\right]-\ZEP
\end{equation}
then $\ZEP=0$.

If~(\ref{eq:PerLidPARALL}) holds then we can find $\tilde u$ and
$\tilde v$, corresponding to the initial states $x+ \xi$ and $x-
\xi$, such that
$$
J( x + \xi;\tilde u)+J(  x - \xi;\tilde v)\leq 2\left[I(x)+I(\xi
)\right]-\ZEP/2\,.
$$
For the initial conditions
$x, \xi$ we apply, respectively, controls
$$
u_0=\frac{\tilde u+\tilde v}{2}\,,\qquad v_0=\frac{\tilde u-\tilde v}{2}.
$$
Then
$$
2\left[ J(  x ; u_0)+J(   \xi;v_0) \right] = J( x + \xi; u_0+v_0)+
J( x - \xi;
 u_0-v_0) \leq 2\left[I(x)+I(\xi )\right]-\ZEP/2\,.
$$
We have also
\begin{equation}\ZLA{eq:ParaIDEforI}
I(x) +I(\xi ) \leq J( x ; u_0)+J( \xi; v_0)\leq \left[ I(x)+I(\xi
)\right]-\ZEP/4\,.
\end{equation}
This shows $\ZEP=0$ so that parallelogram identity holds.

The operator $  P$ is now constructed by polarization
 (compare with~\cite{Halmos}),
\begin{equation}\ZLA{eq:DefiPs}
 \ZL   x,{  P}\xi \ZR
 = I\left (
 \frac{1}{2}(x+\xi)
 \right )-
  I\left (
 \frac{1}{2}(x-\xi)
 \right )
 +i \left [I\left (
 \frac{1}{2}(x+i\xi)
 \right )
 -  I\left (
 \frac{1}{2}(x-i\xi)
 \right )\right ]\,.
\end{equation}
The property
 $
I(x)=\ZL x,Px\ZR
 $
 is a routine computation, using~(\ref{eq:ProBASEiX}).

 We prove property~{\bf a)}.
 Using~(\ref{eq:ProBASEiX})
 we see that the right hand side of~(\ref{eq:DefiPs}) is equal to:
 \begin{multline*}
 I\left (
 \frac{1}{2}(\xi+x)
 \right )-
  I\left (
 \frac{1}{2}(\xi-x)
 \right )
 +i I\left (
 \frac{1}{2}(i\xi+x)
 \right )
 -i I\left (
 \frac{1}{2}(i\xi-x)
 \right )\\
 =I\left (
 \frac{1}{2}(\xi+x)
 \right )-
  I\left (
 \frac{1}{2}(\xi-x)
 \right )
 +i I\left (
 \frac{1}{2}(\xi-ix)
 \right )
 -i I\left (
 \frac{1}{2}(\xi+ix)
 \right )\\
 =\overline{ \ZL \xi,Px\ZR }=\ZL Px,\xi\ZR\,.
 \end{multline*}
In order
 to see property {\bf b)} it is sufficient
 to prove additivity of the
real part. In fact, using    $4I(y_0)= I(2 y_0) $, we check that
\begin{align}
  \zreal \left (4\ZL y,P(\xi+x)\ZR\right )=
  \zreal \left ( \ZL2 y,P2(\xi+x)\ZR\right )
=I(x+\xi+y)-I(x+\xi-y)\\
 =4 \zreal  \left ( \ZL y,P \xi \ZR+\ZL y,P x\ZR\right )
 \label{right}
 = I(\xi+y)-I(\xi-y)+I(x+y)-I(x-y)\,.
\end{align}
Using the parallelogram identity for $ I(x) $, i.e.,~(\ref{paral})
with $ = $ instead of $\le$, and associating the terms of equal
signs, we see that the right hand side of \eqref{right} is equal to
\begin{multline*}
\frac{1}{2}\left [
I(x+\xi+2y)+I(\xi-x)
\right ]-\frac{1}{2}\left [
I(x+\xi-2y)+I(\xi-x)
\right ]\\
=\frac{1}{2}\left [
I(x+\xi+2y)-I(x+\xi-2y)
\right ]\\
=\frac{1}{2}\left \{-I(x+\xi)+
2\left [ I(x+\xi+y)+I(y)\right ]
+I(x+\xi)
-
2\left [ I(x+\xi-y)+I(y)   \right ]
\right \}\\
=I(x+\xi+y)-I(x+\xi-y)
\end{multline*}
as wanted.

Property {\bf c)} for $ q $ real rational is consequence of {\bf
b)}, as in~\cite[Sect.~9.2]{greub}. When $ q=i $ equality follows
since
 {\bf a)}
easily shows
 \[
 \ZL x,P(i\xi)\ZR=-i\ZL x,P\xi\ZR=\ZL x,
 iP\xi\ZR\qquad {\rm i.e.,}\qquad P(i\xi)=iP\xi\,.
  \]
Hence, property {\bf c)} holds also for $ iq $ with real rational $
q $ and then it holds for every complex number with rational real
and imaginary parts.

 Property {\bf d)} is obvious.
 \end{proof}

Now we prove:
\begin{Lemma}\ZLA{teo:USObaire} Assume Hypothesis~\ref{Hyp:0}
and suppose that  system~(\ref{base}) is null controllable. Then,
there exists a number $ M $ such that
\[
 I(y)\leq M|y|^2,\;\;\; y \in H.
 \]
\end{Lemma}
\begin{proof}
It is sufficient to prove that $ I(y) $ is bounded in a ball since
$I (\lambda y_0) = |\lambda|^2 I(y_0)$, $\lambda \in \C$, $y_0 \in
H$. This is known, see the second statement in
Lemma~\ref{Lemma:uniFOcontrolTIME}.
\end{proof}

For the moment, we can't say that the operator $ P $ is linear,
i.e., that $ P(qx)=qPx $ for every \emph{real} $ q $. This will be
proved below, as a consequence of this version of Schwarz
inequality, which can be proved using solely the
  properties stated in Lemmas~\ref{teo:defiFOqquadr}
  and~\ref{teo:USObaire}:

  \begin{Lemma}\ZLA{Lemma:StimaProdoCONp}
  Assume Hypothesis~\ref{Hyp:0}
and suppose that  system~(\ref{base}) is null controllable.
 Then, we
have
 \begin{align} \label{ve}
  |
 \ZL Py,x\ZR |=|\ZL y,Px\ZR
  |\leq M |y | |x|,\;\;\, x,y\in H.
  \end{align}
  \end{Lemma}
  \begin{proof}
 The inequality
is obvious if $ \ZL Py,x\ZR=0 $. Otherwise, we note the following
equality, which holds for every complex number $ \zl $ which has
rational real and imaginary parts:
  \[
  0\leq \ZL Px,x\ZR+2\zreal\left (\lambda \ZL Py,x\ZR\right )+|\lambda|^2\ZL Py,y\ZR\,.
   \]
This inequality is extended to every complex $\lambda$ by
continuity. The usual choice $ \lambda=-\left (\ZL P x,x\ZR\right )
/(\ZL Py,x\ZR) $ gives
\[
|\ZL Py,x\ZR|\leq \sqrt{\ZL Px,x\ZR\ZL Py,y\ZR}=\sqrt{I(x)I(y)}\leq
M |x| |y| \,.
 \]
  \end{proof}

Finally we can prove:
\begin{Lemma} Assume Hypothesis~\ref{Hyp:0}
and suppose that  system~(\ref{base}) is null controllable.
Then the
operator $ P $ defined in Lemma~\ref{teo:defiFOqquadr}
 is linear and continuous on $ H $.
 Hence, it is selfadjoint and non-negative.
\end{Lemma}
\begin{proof}
We first prove that for every complex $ q_0 $ and every $ \xi\in H $ we have
\[
P(q_0\xi)=q_0P\xi\,.
 \]

 Let $ q_n\to q_0 $ be a sequence with \emph{rational} real and imaginary parts. Then we have (Lemma~\ref{Lemma:StimaProdoCONp} is used in the second line)
\begin{eqnarray*}
&& \lim P(q_n\xi)=\lim q_nP\xi=q_0P\xi\,,\\
&& |P(q_n\xi)-P(q_0\xi)|=| P(q_n -q_0)\xi | =\sup _{|y|=1}\ZL y,
P(q_n -q_0 )\xi\ZR
 \leq M |q_n-q_0| \, |\xi|.
\end{eqnarray*}
So, $ q_0P\xi=\lim q_nP\xi=\lim P(q_n\xi)=P(q_0\xi)$. This gives
linearity of the operator $ P $ which, from
Lemma~\ref{teo:defiFOqquadr} is everywhere defined  and symmetric.
Continuity follows immediately from \eqref{ve}.
\end{proof}

 An obvious but important observation is the following one:
the time $ 0 $ as initial time has no special role and we can repeat
the previous arguments, for every initial time $ \tau \ge 0 $ and $
y_0\in H $. Hence we can define $P_{\tau} : H \to H$ such that
 \begin{equation}
\ZLA{eq:INfTAU} \ZL y_0,P_\tau y_0\ZR=\inf  \int_\tau ^{+\infty}|
u(s)|^2\ZD s,
\end{equation}
   where the infimum is computed on the set
   \[
   {\mathcal U}(y_0,\tau)=\{ u\in L^2(\tau,+\infty) \,: \,\, y(t;
   \tau , y_0,u)\in L^2(\tau,+\infty)\}\,.
    \]
       We have a family $ P_\tau $ of linear operators,
 and $ P_0=P $ is the  operator defined in~(\ref{eq:DeFInixDiP}). The observation is:
       \begin{Lemma}
  Assume Hypothesis~\ref{Hyp:0}
and suppose that  system~(\ref{base}) is null controllable.   Then
the operator $ P_\tau $ does not depend on $ \tau $:
  \[
  P_\tau=P_0=P\,.
   \]
   \end{Lemma}
   \begin{proof}
 We observe the following equality, which holds for $ t>\tau $:
  \[
 y(t;\tau,y_0,u)=y(t-\tau;0,y_0,v)\,,\qquad v(t)=u(t+\tau),
   \]
   and
both $ v(t) $ and $  y(t;0,y_0,v) $, $ t\geq 0 $, are square
integrable on $(0, +\infty)$ if $ u(t) $ and  $ y(t;\tau,y_0,u) $
are square integrable on $ (\tau, + \infty) $. Hence, the infimum of
the functional in~(\ref{eq:INfTAU}) is
$ \ZL y_0,P_{\tau} y_0 \ZR  =\ZL y_0,Py_0\ZR$,
      i.e., $ P_\tau=P $.
 \end{proof}

\smallskip

\begin{proof}[\it  Proof of Theorem \ref{teo:KYP}]
We    write
 \[
     \ZL y_0,P y_0\ZR
 \leq \int_0^{+\infty} |u(s)|^2\ZD s=
  \int_0^{\tau} |u(s)|^2\ZD s+ \int_\tau^{+\infty} |u(s)|^2\ZD s\,.
  \]
We choose a control  $ u $ which is smooth on $ [0,\tau] $ so that $
y(\cdot;y_0,u)  $ is continuous (cf. Lemma~\ref{Lemma:ContY}) and we
keep the restriction of $ u $ to $ [0,\tau] $ fixed. The vector
$y(\tau;y_0,u)  $ being controllable to the rest, we can
use~(\ref{eq:ladefiPPRIMAvers}) with initial condition $ \tau $ and
we have
\[
\ZL y(\tau;y_0,u),P_\tau y(\tau;y_0,u)\ZR
=\inf \int_\tau^{+\infty} |u(s)|^2\ZD s
 \]
(as usual, the infimum is computed on those square integrable controls which produces a square integrable solution, on $ [0,+\infty) $).
So, we have
\[
  \ZL y_0,Py_0\ZR\leq \int_0^{\tau} |u(s)|^2\ZD s+\ZL y(\tau;y_0,u),P_\tau y(\tau;y_0,u)\ZR\,.
\]
Using the fact that $ P_\tau=P $ is independent of $ \tau $, we see
that the following inequality holds for every   control $
u  $ which is of class $ C^1 $ on $ [0,\tau] $, every $y_0 \in H$ and $
t\ge 0 $:
  \begin{equation}
\ZLA{eq:INteINEQINTEGRALform1}
\text{(LOI)} \qquad   \ZL  Py(t;y_0,u),y(t;y_0,u)\ZR-\ZL  Py_0,y_0\ZR+
   \int_0^t |u(s)|^2\ZD s\geq 0\,.
\end{equation}
 Finally a standard approximation argument shows
that $(LOI)$ holds even if $u \in L^2_{loc}(0, +\infty;
U)$.
\end{proof}

\medskip

Combining Corollary~\ref{Corollary:IyENCVE}
 and Theorem~\ref{teo:KYP},  we get:
   \begin{Corollary} \label{cia}
 Assume Hypothesis~\ref{Hyp:0}
and suppose that  system~(\ref{base}) is null controllable. Then
System~(\ref{base}) is NCVE if and only if $ P=0 $.
   \end{Corollary}

   These are the preliminaries we need in order to prove
Theorem~\ref{secondo} an equivalent formulation of which is as follows:

\begin{Theorem} \label{secondoFORMAdue}
 Assume  Hypotheses~\ref{Hyp:0} and~\ref{serve}
 and furthermore suppose that $ s(A)\leq 0 $.
  If   system~(\ref{base}) is
  null controllable at some time $T>0$,
then $ P=0 $.
\end{Theorem}
\begin{proof}
 We
decompose $ H=H_s \oplus H_1 $ according to  Hypothesis~\ref{serve}
 and
we show that the restrictions of $P$   respectively to $H_s$ and
$H_1$
 are zero.

If $y_0 \in H_s$, then  by (LOI) with $u=0$ we get
$$
\left \langle  P e^{At} y_0, e^{At} y_0\right  \rangle  \ge \langle  Py_0,y_0\rangle\geq 0.
$$
By assumption, when $ y_0\in H_s $ we have
\[
\lim _{t\to+\infty  }    e^{At} y_0 =0\,.
 \]
Hence, letting
 $t \to + \infty$, we find $\langle  Py_0,y_0\rangle =0$ and so $Py_0=0$.

If $ H_1=\{0\} $, in particular if $ \ZSI(A)=\emptyset $, then $
H=H_1 $ and we are done. Otherwise, we prove that $ P $ is zero on $
H_1 $.

In order to prove that $ P $ is zero on $ H_1 $, \emph{ it is enough to verify  that
$ P z=0 $ on every generalized eigenvector $z$ of $ A $ which
belongs to $ H_1 $.\/}

Indeed the subspace generated by all the generalized eigenvectors of $ A
$ which belong to $ H_1 $ is dense  in this space and,
 moreover,  $P$ is continuous on $H_1$.

We recast the definitions of the generalized eigenvectors in the
form that we need. Let $ y_0 $ be an eigenvector of $ A $, $
Ay_0=\zl y_0 $. We associate to $ y_0 $ the ``Jordan chain'' whose
elements are the vectors $ y_k $ which, for $ k\geq 1 $, are defined
by
\begin{equation}
\ZLA{eq:DefiJORDANchain}
 A y_k=\zl y_k -  y_{k-1}\,.
\end{equation}
 This process ends at the index $ n $
if $y_{n+1}=0  $. So, a Jordan chain may be infinite or finite
(possible of length $ 1 $, reduced to $ y_0 $). The generalized
eigenvectors of $ A $ are the elements of a Jordan chain.

Note that the chain is identified by both $ \zl $ and its
eigenvector $ y_0 $: an eigenvalue $ \zl $ has as many (independent)
Jordan chains as   independent eigenvectors.

 A fact we shall use is that when $ y_k $ is an element of a Jordan chain
 (corresponding to an eigenvalue $ \zl $) then
\begin{equation}
\ZLA{eq:expone}
 e^{At} y_k= e^{\zl t}y_k+ q_k(t)\,,\qquad   \;
 q_k(t)=e^{\zl t}\sum _{n=1}^k \alpha_{k,n} y_{k-n}t^n  \;
\end{equation}
for suitable coefficients $  \alpha_{k,n}$. The important point to be noted is that \emph{ only the elements $ y_r $ of the chain, with $ r<k $, appear in the expression of $ q_k(t) $.}

Now we prove that $  Py=0$ if $ y $ is any generalized eigenvector of $ A $ in $ H_1 $. We distinguish two cases, that the eigenvalue has negative real part, or null real part (positive real part is impossible, due to our assumption $ s(A)\leq 0 $).

 \subsubsection{the case $ \zreal\zl<0 $}
Let $ y_N $ be a generalized eigenvector, defined by the sequence of
the equalities~(\ref{eq:DefiJORDANchain}) and let
\[
Y_N={\rm span}\{y_k\} _{0\leq k\leq N}\,.
 \]
 The subspace $ Y_N $ is invariant for $ A $
and there exist \emph{positive} number $ M $ and $ \sigma $ such that
\[
y\in Y_N\ \implies \left  | e^{At} y\right  |\leq M e^{-\sigma t} |y |\qquad \forall y\in Y_N\,.
 \]
 So, the same argument as used above for the case of the subspace $ H_s $ can be used here: (LOI) with $ u=0 $ gives
 \[
0\leq  \langle Py,y\rangle \leq \langle Pe^{At  }y,
e^{At}y\rangle\leq M^2 e^{-2\sigma t}\| P\|\, |y|
  \]
  and the right hand side tends to $ 0 $ for $ t\to+\infty $.

  This proves that $ P=0 $ on $ Y_N $, as wanted, when the corresponding eigenvalue has negative real part.

  \subsubsection{the case $ \zreal\zl=0 $}
We recall the notation $ (Lu)(t) $ from~(\ref{eq:TrasfCONTROevolNELduale}) so that
\[
y(t;y_0,u)=e^{At} y_0+(Lu)(t)
 \]
and (LOI) can be written as
\begin{align}
\nonumber
\left \{
\left \langle e^{At} y_0,Pe^{At} y_0\right \rangle -\left \langle y_0,Py_0\right \rangle
\right \}+2\zreal \left \langle (Lu)(t),Pe^{At} y_0\right \rangle \\
\ZLA{eq:LOIconLambda} +\left \{ \intt \|u(s)\|^2\ZD s+\left \langle
(Lu)(t),P(Lu)(t)\right \rangle \right \}\, \ge 0\,.
\end{align}
Now let $ y_0 $ be   any eigenvector   of the eigenvalue $ i\omega$,
$ \omega\in\R$. Then
\[
\left \langle e^{At} y_0,Pe^{At} y_0\right \rangle=\left \langle
e^{i\omega t} y_0,Pe^{i\omega t} y_0\right \rangle=\left \langle
y_0,P  y_0\right \rangle
 \]
and the first brace in~(\ref{eq:LOIconLambda}) is equal $ 0 $.
 Replacing $y_0$ by $\mu y_0$,
$\mu \in \R$, we see that
 \[
   \langle P (L u)(t),y_0\rangle=0.
    \]
So, for every $ u\in L^2_{\rm loc}(0,+\infty) $ we have, for a.e. $
t \ge 0 $,
\begin{align} \label{ses}
    \left \langle y_0,P\left [ e^{A t}y_0+(Lu)(t)\right ]\right \rangle=
        \left \langle y_0,P\left [ e^{i\omega t}y_0+(Lu)(t)\right ]\right \rangle=
  e^{- i\omega t}
   \langle y_0,Py_0\rangle.
     \end{align}
The system being controllable to the rest at time $ T $, there exists a control such that
\[
e^{A t}y_0+(Lu)(t)=e^{i\omega t}y_0+(Lu)(t)
 \]
 has support in $ [0,T] $, and so  the left hand side of~(\ref{ses}) is zero   for $ t>T $. Hence,   $ \langle y_0,Py_0\rangle=0 $. As $ P=P^*\geq 0 $, we see that $ Py_0=0 $, as wanted.

 Now we extend this property to every element of the Jordan chain of $ y_0 $, using an induction argument.
 Let $  y_N$ be a generalized eigenvector of this chain and let us assume that $ Py_k=0 $ for $ k<N $. We prove that $ P y_N=0 $ too.

 Using formula~(\ref{eq:expone}), we see that the induction hypothesis implies
 \[
 P q_N(t)=0
  \]
 and combining (LOI),~(\ref{eq:LOIconLambda}) and~(\ref{eq:expone}) we get
 \[
 2\zreal e^{i\omega t}\left \langle P y_N,(L u)(t)\right \rangle
 +\left \{
\intt \|u(s)\|^2\ZD s+\left \langle (L u)(t),P(L u)(t)\right \rangle
\right \}\geq 0\,.
  \]
  As above, the part which is linear in $ y_N $ has to be zero,
  i.e.,
   $
  \left \langle P y_N,(L u)(t)\right \rangle=0
$,
   so that
   \[
    \left \langle y_N, P y(t;y_N,u)\right \rangle=
   \left \langle y_N, P\left (e^{At}y_N+(L u)(t)\right )\right \rangle=\left \langle Py_N,
   e^{i\omega t}y_N\right \rangle\,.
    \]
    We then control   $ y_N $  to the rest   using a suitable control $ u $ and, as above, we get $ Py_N=0 $. This ends the proof.
    \end{proof}

\section{Examples and applications}

Here we provide applications of our  result to boundary
control of parabolic coupled equations considered  in \cite{CARA} and to delay systems. The problem of finding explicit conditions on the minimal energy for large times is discussed in some details as well.

\subsection{Parabolic coupled system}

We will establish  {\it conditions under which the control system of
\cite{CARA} is NCVE.} We start from establishing some properties of the system.

The system is linear and of the form:
\begin{align}\label{cara}
\begin{cases}
y_t - y_{xx} = A_0 y \;\;\; \text{in}\; Q= (0,\pi ) \times (0,T),
\\
y(0, \cdot)= B_0 u,\;\;\;\; y(\pi ,0)=0 \;\; \text{in} \; (0,T),
\\
y(\cdot, 0)= y_0 \; \text{in} \; (0,\pi ),
\end{cases}
\end{align}
where $A_0 $ is a $2 \times 2$ real matrix and $B_0 \in \R^2,$ $u
\in L^2(0,T)$ is the control function and $y =
 {\rm col}\,\left(\begin{array}{cc}y_1, y_2\end{array}\right)$ is the
state variable.   As in \cite{CARA} we denote by
$$
L^2(0,\pi)^2
$$
the space $(L^2(0,\pi))^2$ and use similar notation  for other function spaces as well.

In \cite{CARA} the  interval $(0, \pi)$
 is replaced by $(0,1)$ and the
initial datum $y_0 \in
 H^{-1}(0,\pi)^2$  (the dual space of the space $H^1_0(0,\pi)^2=
H^1_0(0,\pi; \R^2)$) but in the sequel we will assume that
 $y_0 \in L^2(0,\pi)^2$.

Let $y_0 \in L^2(0,\pi )^2$ and $u \in L^2(0,T)$, in \cite{CARA} it
is proved that there exists a unique   solution   $y \in L^2(Q)^2$
to \eqref{cara}. This is defined by transposition, i.e., requiring
that, for each $g \in L^2(Q)^2$ one has

\begin{align}\label{cara1}
    \int_Q y \cdot g \, dx dt  =
\langle y_0, \phi(\cdot, 0) \rangle + \int_0^T B_0 \cdot
\phi_x(0,t)u(t) dt,
\end{align}
where $\langle \cdot, \cdot \rangle$ is the inner product in
 $L^2(0,\pi )^2$
and
$$
\phi \in L^2(0,T; H^{2}(0,\pi )^2) \cap C^0 ([0,T]; H^{1}_0(0,\pi
)^2)
$$
is the (unique) strong solution to the following equation involving $g$
$$
\begin{cases}
-\phi_t - \phi_{xx} = A_0^* \phi + g \;\;\; \text{in}\; Q= (0,\pi )
\times (0,T),
\\
\phi(0, \cdot)= 0,\;\;\;\; \phi(\pi ,0)=0 \;\; \text{in} \; (0,T),
\\
\phi(\cdot, T)= 0 \; \text{in} \; (0,\pi ).
\end{cases}
$$
It is not difficult to see that {\it the previous   solution $y$
can be written in the form \eqref{eq:SoluBASE2}.} In addition we see
that Hypotheses \ref{Hyp:0} and \ref{serve}  hold in this case. In
fact:

\begin{Proposition} \label{c} The system \eqref{cara}
can be written in the form \eqref{base} and in addition Hypotheses \ref{Hyp:0} and \ref{serve} hold.
\end{Proposition}
\begin{proof} We fix $H = L^2(0,\pi ; \C )^2$
and define the linear operator
$A: \text{dom} \,A  \subset H \to H$,
$$
\text{dom} \,A  = H^{2}(0,\pi ; \C)^2 \cap H^{1}_0(0,\pi ;
\C)^2,\;\;\; Af =  f_{xx} + A_0 f,\;\; f \in \text{dom}\,A\,.
$$
Since the operator  $A_1 = D_{xx}$ with
 $\text{dom}\,A_1 =\text{dom}\, A $
generates a compact holomorphic  semigroup on $H$ the same happens
for $A$ (even if, in general, $e^{At}$ is no more symmetric).
Indeed
$A_0$ is a  bounded perturbation  and we can apply
well-know results in \cite[Section 3.1]{Pa}.

 In particular the resolvent operator of $A$ is
 compact  and
 the spectrum $\sigma(A)$  consists entirely of eigenvalues
with finite algebraic multiplicity; this also shows that Hypothesis \ref{serve}
can be used.

\vskip 1mm We show that when $y_0 \in L^2(0,\pi )^2$ the unique
solution  of \eqref{cara}
$$
y \in L^2(Q)^2 = L^2 (0,T; L^2(0,\pi )^2)
$$
 defined by transposition
  can be written as in \eqref{eq:SoluBASE2} and so it
coincides with the solution to \eqref{base}.

 Note that at least for regular real functions $y_0 \in
 C^1([0,\pi])^2$ and $u \in C^1([0,T])$ one checks directly that the solution
$y= y^{y_0,u} $ to \eqref{cara} is given by
\begin{align}\label{for1}
y(t) = e^{At}y_0  -  A \int_0^t e^{A(t-s)} C u(s)\ZD s,
\end{align}
where, $C: \C \to C^2([0,\pi ])^2$, $h(x)=(Ca)(x)$, $x \in [0,\pi
]$, $a \in \C,$ is the unique   solution to
$$
\begin{cases}
 h'' + A_0 h= 0 \;\;\; \text{in}\; (0,\pi ),
\\
h(0)= aB_0,\;\;\;\; h(\pi )=0,
\end{cases}
$$
  for any $a \in \C$.

 Let us introduce the control operator $B$.
Setting $U = \C$, we define $
 B = -{\mathcal A} C, $
where ${\mathcal A}$ is the extension of $A$ to $({\rm dom}\, A^*)^{'}$, i.e.,
 ${\mathcal A}: H \to \text{(dom$\,A^*)^{'}$}$. Hence,
 we can rewrite  \eqref{for1} as
 \begin{align}\label{ba3}
  y(t) = e^{At}y_0  + \int_0^t e^{\A(t-s)} B u(s)\ZD s.
 \end{align}
 In order to check Hypothesis \ref{Hyp:0}  we
 first clarify that, for any $u \in L^2(0,T; \C)$, the
mapping:
\begin{align} \label{tes4}
t \mapsto -  A \int_0^t e^{A(t-s)} C u(s)\ZD s =
  \int_0^t
  e^{{\mathcal A}(t-s)} B u(s)\ZD s \;\; \text{belongs to $L^2(0,T;H)$}.
\end{align}
 Since $e^{At}$ is holomorphic, for any
 $\theta \in (0,1/4)$, for any
$f$ which belongs to the interpolation space
$$
(H, \text{dom}\,A )_{\theta, \infty} =
(H, \text{dom} \,A_1 )_{\theta, \infty}
$$
(see, for instance, \cite[pag. 148]{BD1})
we have $\| A e^{At} f\|_{ H} \le \frac{C_T}{t^{1-\theta}}\,  \| f\|_{(H, \text{dom} \,A)_{\theta, \infty} } $, $t \in (0,T)$. On the other hand, it is well known that
$$
 \text{dom} [(-A_1)^{\theta}] = (H, \text{dom} \,A_1 )_{\theta, \infty} =
  H^{2 \theta}(0,\pi ; \C)^2
$$
(with equivalence of norms). Since, in particular,
the operator
$$
C : \R \to  H^{2 \theta}(0,\pi ; \C)^2
$$
is continuous, it follows that  $\| A e^{At} C a\|_{ H} \le
\frac{C_T \, |a|}{t^{1-\theta}}$, $t \in (0,T)$, $a \in \C$. Now, by
the Young inequality for convolution we deduce easily  \eqref{tes4}.
Moreover, we also obtain that  the transformation $u \mapsto
\int_0^t
  e^{{\mathcal A}(t-s)} B u(s)\ZD s$ is linear and bounded from
   $L^2(0,T; \C)$ into $L^2(0,T; H)$ and
 so  Hypothesis \ref{Hyp:0} is satisfied.

\vskip 1mm It remains to show that for  $y_0 \in L^2(0,\pi )^2 $ and
$u \in L^2(0,T)$, the function given in \eqref{ba3} is the  weak
solution to \eqref{cara}.

\vskip 1mm  By \eqref{tes4}
 we know that $y = y^{y_0,u}\in L^2(0,T; L^2(0,\pi )^2)$.
 Choosing regular $(z_n)$ and $(u_n)$ converging to $y_0$ and $u$
respectively in $L^2(0,\pi )^2$ and $L^2(0,T)$ we note that the
solutions $y_n = y^{z_n, u_n}$ verify the identity \eqref{cara1}.

 Since $y_n$ converges to $y$ in $L^2(0,T;L^2(0,\pi )^2)$, passing to the limit
 as $n \to \infty$ in \eqref{cara1} we deduce that $y$ verifies
 \eqref{cara1}.
\end{proof}

\subsection{Null controllability with vanishing energy.}
In \cite{CARA} it is proved that \eqref{cara} is null controllable
at any time $T>0$ if and only if
\begin{equation}\ZLA{eq:CondiCONTRLLf-C}
 \text{rank$[A_0| B_0]= \text{rank}[B_0, A_0 B_0] = 2$}
  \;\;\; \text{and} \;\;\;
{\mu_1 - \mu_2}
\not = j^2 - k^2, \;\; \forall k,j \in
\mathbb{N},
\end{equation}
with $j \not =k$, where $\mu_1$ and $\mu_2$ are the eigenvalues of
$A_0$.

We can prove the following additional result.

\begin{Theorem}\label{teoControllCARA}
Assume that the control system \eqref{cara} is null controllable
(i.e., that condition~(\ref{eq:CondiCONTRLLf-C}) hold). Then the
system \eqref{cara} is NCVE if and only if
\begin{align}\label{se4}
\zreal (\mu_i) \le  1,\;\; i =1,2.
\end{align}
\end{Theorem}
\begin{proof} By combining    Theorems~\ref{primo}
 and~\ref{secondo} (see also
Corollary~\ref{coro:rema:nounif}) we see that
  \eqref{cara} is NCVE if and only if
\begin{align*}
    \sigma (A) \subset \{ \zreal\zl \leq 0  \}.
\end{align*}
Thus \eqref{se4} follows easily if we show that
\begin{align}\label{spe5}
    \sigma (A) = \{ \lambda \in \C \; : \; \lambda =
    \mu_i  -k^2, \;\;\; i =1,2,\;\; k \in  \mathbb{N},\; k \ge 1
    \}.
\end{align}
To check \eqref{spe5} we first recall
 that   $\sigma (A)$ consists entirely of
  eigenvalues (see the proof of Proposition \ref{c}).

 Moreover, we will use that  $A_D = D_{xx}$ with Dirichlet boundary condition
 is self-adjoint on $L^2(0,\pi )$ with $\sigma
 (A_D) = \{ -k^2  \}_{k \ge 1}$ and
 \begin{align} \label{lap2}
  A_D e_k = -k^2  e_k,\;\;\; e_k(x) =  \frac{\sqrt{2}}{\sqrt{\pi}}
  \sin(k  x),\;\; x \in [0,\pi ].
 \end{align}
 In order to characterize $\sigma(A)$, we
 fix an
 eigenvalue $\lambda \in \C$ and consider a corresponding
eigenfunction $u \in \text{dom}\,A $,
 i.e.,
\begin{align}\label{eig}
u_{xx} + A_0 u = \lambda u,\;\; \; u \not = 0.
\end{align}

We note that if $ A_0 $ is diagonalizable with only one (repeated)
eigenvalue, then rank$[A_0| B_0]=1$. So, controllability implies
that we have to examine only the following two cases:

\vskip 1mm  \noindent (i) $A_0$ has a unique (real) eigenvalue $\mu = \mu_1 = \mu_2$ with dim(Ker$(A_0-\mu))=1$.

We introduce a non-singular $2\times 2$ real matrix $P$ such that
$$
P^{-1} A_0 P = J =  \left(\begin{array}{cc}
\mu  & 1 \\
0 & \mu
\end{array}
\right)
$$
and consider the real function $v = P^{-1}u \in \text{dom}\,A$. We find
 $Pv_{xx} + A_0 Pv  = \lambda Pv,$ and so we can concentrate on the problem
$$
v_{xx} + J v = \lambda v.
$$
If $v ={\rm col}\, (v^{(1)}, v^{(2)}) $ we find
$$
\begin{cases}
v_{xx}^{(1)} + \mu v^{(1)} + v^{(2)} = \lambda v^{(1)}.
\\
v_{xx}^{(2)} + \mu v^{(2)} = \lambda v^{(2)}.
\end{cases}
$$
 Using \eqref{lap2},  we deduce  that
 $\lambda - \mu = - k^2 $ for some $k \ge 1.$ Moreover $u = Pv$
where $v = {\rm col}\,( e_k,0) $ is an eigenfunction corresponding
to $\lambda =  \mu - k^2 $.

The eigenvalue $ \zl=\mu-k^2 $ has a Jordan chain of length $ 2 $
and the generalized eigenvalue has both the components equal to $
e_k $. Hence, the Jordan chain of $ \zl_k=\mu-k^2 $ has the
following elements (with
 $c = \frac{\sqrt{2}}{\sqrt{\pi}}$):
\[
\left (
\begin{array}
{cc} c \sin (k x)\\
0
\end{array}
\right )\,,\qquad
\left (
\begin{array}
{cc}c \sin (k x)\\
c \sin (k x)
\end{array}
\right )
 \]
and so the Jordan chains span the state space, in spite of the fact that the operator is not selfadjoint.

  \vskip 1mm  \noindent (ii)  $A_0$ has distinct eigenvalues $\mu_1$ and  $ \mu_2$, which might be complex (conjugate).

We consider a non-singular $2\times 2$   matrix $P$ (possibly complex) such that
$$
P^{-1} A_0 P = J =  \left(\begin{array}{cc}
\mu_1  & 0 \\
0 &   \mu_2
\end{array}
\right).
$$
Introducing  the complex function $v = P^{-1}u$, we find $ v_{xx} + J v  = \lambda v. $
If  $v = {\rm  col}\, ( v^{(1)},v^{(2)} )
$
  we get
$$
\begin{cases}
 v_{xx}^{(1)} + \mu_1 v^{(1)}  = \lambda v^{(1)}
\\
v_{xx}^{(2)} +   \mu_2 v^{(2)} = \lambda v^{(2)}\,.
\end{cases}
$$
 Using \eqref{lap2},  we deduce  that, for  $k, n \ge 1,$
$$
\text{either} \;\; \lambda - \mu_1 = - k^2 \;\; \text{or} \;\;
\lambda -  \mu_2 = - n^2.
$$
Moreover, if $\lambda =  \mu  - k^2 $, then
 an eigenfunction is  $u = Pv$ where
$v ={\rm  col}\, ( e_k,0) $. If $\lambda =    \mu_2  - n^2 $, then
 an eigenfunction is  $u = Pv$ where
$v = {\rm col} ( 0, e_n) $ (the functions $ e_n $ are defined
in~(\ref{lap2})). Hence, there is an  orthonormal basis of
eigenvectors of the state space in this case, whose elements are
\begin{equation}
\ZLA{eq:baseAUTOVcasoDIAG}
\left (
\begin{array}
{cc}c \sin (n x)\\
0
\end{array}
\right )\,,\qquad
\left (
\begin{array}
{cc} 0\\
c \sin (k x)
\end{array}
\right )
\end{equation}
 where $ n \in {\mathbb N}$ and $ k \in {\mathbb N}$ are independent.

The proof is complete.
  \end{proof}

\subsection{\ZLA{sect:ExplEstimNCVE}An explicit
 computation of the control energy }

Theorem~\ref{teoControllCARA} provides necessary and sufficient
conditions under which system \eqref{cara} is NCVE. Recall that if
 NCVE holds then, for every $ \ZEP>0 $, there exist controls steering
to the rest any initial condition and whose norm is less then $ \ZEP
$. Thus  if the system is  NCVE
 it might  be of interest
  to estimate
 the control  energy at time $ T $ , i.e.
\[
Z_T(y_0)
 =  \inf    \int_0^{T}
  | u(s) | ^2  ds
 \]
(the infimum is computed on the controls steering $ y_0 \in H$ to
the rest in time $ T >0$) and show directly that this converges to 0
as $T \to + \infty$.  For an example of such computations in the
case of the wave equation with boundary controls, see~\cite{ivanov};
see also  \cite{PWX}  for a related time optimal control problem  in
the {distributed} control case.
 We cite also ~\cite[Sect.~4.3]{ZUAZUA} for a general discussion on parabolic systems.

 \emph{We are going to show that explicit estimates on the control energy  are indeed
possible in the case of the example~(\ref{cara}), of course at the expenses of
 some more computations. }

 We confine ourselves to consider the case that $ A_0 $ is
 diagonalizable and $s(A) \le 0$. Moreover,
   we put ourselves in the critical case that $ 0 $
is an eigenvalue, so that, after a coordinate transformation,
\[
A_0=\left[\begin{array}
{cc} 1&0\\
0 &\mu
\end{array}\right]
 \]
 with $ \mu\leq 1$  and both $ \mu \not = 1 $ and
 $ \mu\neq 1-(j^2-k^2)  $
for every $ j $ and $ k \in \mathbb{N}$, in order to have
controllability. Furthermore, it is not restrictive that we assume
\[
B_0=\left (
\begin{array}{cc}
1\\ \beta
\end{array}
\right )
 \]
($ \beta\neq 0 $ is required by controllability).  Let, for $ n\geq 1 $,
\[
y(x,t)=\left (
\begin{array}{l}
v(x,t)\\w(x,t)
\end{array}
\right )\,,\qquad\left\{\begin{array}{l} v_n(t)=
 \frac{\sqrt{2}}{\sqrt{\pi}}
\int_0^\pi v(x,t) \sin nx\ZD x\,,\\  w_n(t)=
\frac{\sqrt{2}}{\sqrt{\pi}}\int_0^\pi w(x,t)\sin nx\ZD x\,.
\end{array}\right.
 \]
 An integration by
parts shows that $ v_n(t) $ and $ w_n(t) $ solve the following
equations
\begin{equation}
\ZLA{eq:DiVnWn}\left\{\begin{array}{lll}
 v_n'(t)&=& (1-n^2)v_n(t)+n u(t), \\
 w_k'(t)&=& (\mu-k^2)w_k(t) +k\beta u(t).
 \end{array}\right.
\end{equation}
Note that $ n $ and $ k $ here are independent, since an orthonormal basis of the state space is~(\ref{eq:baseAUTOVcasoDIAG}). So, the control $ u(t) $ steers    $ v(x,0)=v_0(x) $ and $ w(x,0)=w_0(x) $ to the rest in time $ T $ when $ f(t)=u(T-t) $ solves the moment problem (with $ n, \, k \geq 1 $)
\begin{equation}\ZLA{eq:probleMOMENT}
\left\{\begin{array}{lll}
 \intT e^{-(n^2-1)s }f(s)\ZD s&=&\left
 [\frac{1}{n}e^{-(n^2-1)T}\right ] v_{0,n} \,,
\\
\intT e^{-(k^2-\mu)s }f(s)\ZD s&=&\left [\frac{1}{\beta
k}e^{-(k^2-\mu)T} \right ]w_{0,k} \,.
\end{array}\right.
 \end{equation}
 Here $ v_{0,n} $ and $ w_{0,k} $,  $n, k \ge 1$, are the Fourier
coefficients in the sine expansion of the initial conditions $
v_0(x) $ and $ w_0(x) $.

 Let us define the following sequence $ \{\Phi_n(t)\} _{n\geq 1} $:
 \[
 \Phi _{2r-1}(t)=e^{-(r^2-1)t}\,,\qquad
 \Phi _{2r }(t)=e^{-(r^2-\mu)t}
  \]
 for every natural number $ r \ge 1$.  Set also
$$
\lambda_{2r-1}= r^2 -1\,,\qquad \lambda_{2r}= r^2 -\mu
$$
 i.e., for $ n\geq 1 $,
\begin{equation}\ZLA{eq:gliESPONENti}
\lambda_n=\left\{
\begin{array}{ll}
\frac{1}{4}(n+1)^2-1=\frac{1}{4}n^2+\left (\frac{1}{2}n-\frac{3}{4}\right ), &\mbox{$ n $ odd,}\\
\frac{1}{4}n^2-\mu,&\mbox{$ n $ even.}
\end{array}
\right.
\end{equation}
So, $ \Phi_1(t)=1 $ corresponds to the eigenvalue $ \zl_1=0 $ of $A$. Properties of this sequence have been studied in~\cite{Schw} on the space $
L^2(0,+\infty) $ and in $ L^2(0,T) $ (the fact that $ \Phi_1 $ is not square integrable is easily adjusted, see also below, in the proof of Lemma~\ref{Lemma:sullaNORMAbiortog}).
 It is proved   that the sequence $ \{\Phi_n(t)\} $ has a biorthogonal sequences in $ L^2(0,T) $.    Furthermore, from~\cite{AVDivan},  $ f(t) $ is given by
\begin{eqnarray}
\nonumber&&f(t)=\Psi_1^T(t)v_{0,1}+ h^T(t)\,,\\
\ZLA{eq:SOLUprobleMOMENT} &&h^T(t)=\sum_{r \ge 2}
\Psi_{2r-1}^T(t)\frac{1}{r}e^{-(r^2-1)T}v_{0,r}+
 \sum_{r \ge 1}
 \Psi_{2r }^T(t)\frac{1}{\beta r}e^{-(r^2-\mu)T}w_{0,r}\,.
 \end{eqnarray}
where $\{ \Psi_n^T\} _{n\geq 1} $ is any biorthogonal sequence such that the series converges in $ L^2(0,T) $. The existence of this sequence is consequence of the following lemma, proved at the end of this section:

 \begin{Lemma}\ZLA{Lemma:sullaNORMAbiortog}
 For every $ \ZEP>0 $ there exists a biorthogonal sequence $\{ \Psi_n^T\} _{n\geq 1} $ and a  number $K(\ZEP)$, independent of $ T $, such that
 \[
| \Psi_{n }^T|_{L^2(0,T)}<K(\ZEP) e^{\ZEP\zl_n}
  \]
 \end{Lemma}

We assume this lemma and we prove convergence of the series  in~(\ref{eq:SOLUprobleMOMENT}). Furthermore we give an estimate for $ |h^T|_{L^2(0,T)} $.
  We prove convergence for $ T>2 $ since this is all that we need for the asymptotics of the energy. Convergence for $ T\in(0,2] $ is proved analogously.

  We consider the first series, which   can be treated as follows. We fix $ \ZEP=1/2 $ and we denote $ K=K(1/2) $. Then we have:
 \begin{eqnarray*}
&& \left |
\sum_{r \ge 2}
\Psi_{2r-1}^T(t)\frac{1}{r}e^{-(r^2-1)T}v_{0,r}
 \right | _{L^2(0,T)}
 \leq K \sum_{r \ge 2} \frac{1}{r} e^{-(r^2-1)(T-1/2)}|v_{0,r}|\\
 &&\leq Ke^{-(1/2)(T-1/2)}\left[
  \sum_{r \ge 2} \frac{1}{r} e^{-(r^2-3/2)(T-1/2)}|v_{0,r}|
 \right] \\
 &&
  \leq Ke^{-(1/2)(T-1/2)}\left[
  \sum_{r \ge 2} \frac{1}{r} e^{-(r^2- 3/2) }|v_{0,r}|
 \right]
 \end{eqnarray*}
 (since $ T>2 $).
 These inequalities
 prove in one shot convergence of the series and furthermore they prove that
 \[
  \left |
\sum_{r \ge 2}
\Psi_{2r-1}^T(t)\frac{1}{r}e^{-(r^2-1)T}v_{0,r}
 \right | _{L^2(0,T)}\to 0
  \]
 exponentially fast for $ T\to+\infty $.

 The second series can be treated analogously.

Hence, \emph{if $ v_{0,1} =0$ (or if $ \zl=0 $ is not an
eigenvalue) then   $ Z_ T(y_0) $  decays exponentially for $
T\to+\infty $.}

 We consider now $ Z_ T(y_0) $ when $ v_{0,1} \neq 0$. We keep the same elements $ \Psi_n^T $ as above for $ n>1 $, so that  the exponential estimate on the series is not affected, and we choose a suitable element $ \Psi_1 $.

 Note that we can confine ourselves to give an estimate for
$ Z_ T(y_0) $ when $ T=N$, a positive integer.

Let $ \{\Psi_n^1(t)\} $ be the biorthogonal sequence when $ T=1 $
and let us consider its first element $ \Psi_1^1(t) $. It satisfies
\[
\int_0^1 \Psi_1^1(t)\Phi_n(t)\ZD t=\left\{\begin{array}
{lll}
1 &{\rm if}&n=1\\
0&& {\rm otherwise}\,.
\end{array}
\right.
 \]
Let now $ \gamma $ be any positive number. For every $ \Psi
\in L^2(0,N)$  we have
\begin{equation}\ZLA{eq:laestINTperiteraz}
\int_0^N \Psi (t) e^{-\gamma t}\ZD t=\sum _{m=0} ^{N-1}e^{-\gamma
m}\int _0^1 \Psi(m+t)e^{-\gamma t}\ZD t\,.
 \end{equation}
Let us define $ \Psi_1^N(t) $ on $ [0,N] $\,,
\[
\Psi_1^N(t)= \Psi _{1}^1(t-m)\qquad \mbox{if $ m\leq t<m+1 $}, \qquad
0 \le m \le N-1 \,.
 \]
Then,~(\ref{eq:laestINTperiteraz}) shows that  $ \Psi_1^N(t) $ is
orthogonal to   $ \Phi_n(t) $ on $ [0,N] $, for every $ n>1 $, and
\[
\|\Psi_1^N\| _{L^2(0,N)}= {\sqrt N} \|\Psi _ 1^1\|_{L^2(0,1)}  \,.
 \]
 Now we consider the control function $ f(t) $ on $ [0,N] $ given by
\begin{eqnarray*}
&& f(t)= \alpha \Psi_1^N(t)+ h^N(t)\,,\\
&& h^N(t)=\sum_{r \ge 2}
\Psi_{2r-1}^N(t)\frac{1}{r}e^{-(r^2-1)N}v_{0,r}+
 \sum_{r \ge 1}  \Psi_{2r }^N(t)\frac{1}{\beta r}e^{-(r^2-\mu)N}w_{0,r}
 \,,
\end{eqnarray*}
for a suitable constant $\alpha $ to be fixed.

Note that $ h^N(t) $ is the same as in~(\ref{eq:SOLUprobleMOMENT}),
with $ T=N $. Thanks to the fact that $\Psi_1  ^N$ is orthogonal to
every $ \Phi_k(t) $, $ k>1 $, we see that this function $ f(t) $
solves the moment problem~(\ref{eq:probleMOMENT}) (with $ T=N $) if
\[
v_{0,1}-\int_0^N h^N(s)\ZD s=\alpha \int_0^N \Psi_1^N(t)\ZD t=N\alpha
\int_0^1 \Psi_1^1(t)\ZD t=N\alpha\,.
 \]
 We have seen that the $ L^2(0,N) $-norms of $ h^N(t) $ decay exponentially:
 $
 \| h^N\| _{L^2(0,N)}\leq  e^{-\sigma N}
  $ ($ \ZSI>0 $).
 This shows that $ \alpha  \asymp 1/N$ and we have
  \[
  \| \alpha \Psi_1^N\| _{L^2(0,N)}\asymp \frac{{\rm const}}{\sqrt N}\qquad {i.e.}\qquad  Z_N (y_0)\asymp \frac{{\rm const}}{\sqrt N}\,.
   \]
  This is the required estimate for the energy at time $ T $. This estimate implies, in a less direct way, that  the system is   NCVE.


\begin{proof}[ Proof of Lemma~\ref{Lemma:sullaNORMAbiortog} ]
The proof consists on finding a relation between the norm of biorthogonals  in $ L^2(0,T) $ and in $ L^2(0,+\infty ) $. But, we note that $ \Phi_1(t)=1 $ is not square integrable. This is adjusted replacing $ \Phi_n(t) $ with $ e^{-t}\Phi_n(t) $, $ \zl_n $  in~(\ref{eq:gliESPONENti})  with $ \zl_n + 1 $    and then $ f(t) $ with $ e^{t} f(t) $. Once this has been understood, we go on using the notations $ \zl_n $ and $ \Phi_n(t) $ for these (modified) sequences.

The sequence of the exponents $ \lambda_n $ (now all positive) satisfies the conditions in~\cite[Lemma~3.1]{CARA}  (it does not satisfy the more stringent conditions in~\cite{F-R}).
Hence (see~\cite[Lemma~3.1]{CARA}), it admits a biorthogonal sequence $ \{\Psi_n(t)\} $ in $ L^2(0,+\infty) $ with the following property: for every $ \ZEP>0 $ there is a constant $ K(\ZEP) $ such that
\[
|\Psi_n(t)|_{L^2(0,+\infty)}\leq K(\ZEP)e^{\ZEP\zl_n}\,.
 \]
  The sequence $ \zl_n $ in particular satisfies
\[
\sum_{n \ge 1}\frac{1}{\zl_n}<+\infty
 \]
 so that $ \{ e^{-\zl_n t}\} $ spans a closed \emph{proper} subspace of $ L^2(0,+\infty ) $.
Let $ E(\infty) $
and $ E(T) $ be   the closed linear spans of the set $ \{ e^{-\zl_n t}\} $ respectively in $ L^{2}(0,+\infty) $ and $ L^{2}(0,T) $.
Let $ P_T $
 be the linear operator which assigns to any element of $ E(\infty) $ its restriction to $ (0,T) $ (which is an element of $ E(T) $).
 Clearly, $ P_T $ is linear and continuous (norm equal $ 1 $) and, from~\cite[p.~55]{Schw}, it is boundedly invertible,
 \[
 \| P_T^{-1}\|\leq M_T
  \]
(the inverse is defined on $ E(T) $).
  It
follows from~\cite[p.~55]{Schw}  that   for arbitrary positive $T$ there exists a constant $C(T)$  such that for
arbitrary real numbers $a_n$ and arbitrary natural number $N$
\begin{equation}\label{embedding}
\left \|\sum_{n=1}^N a_n e^{-\lambda_n t}  \right \|_{L^{2}(0,+\infty)}\leq C(T) \left \|\sum_{n=1}^N a_n e^{-\lambda_n t}\right \|_{L^{2}(0,T)}.
\end{equation}
Of course, $ T\to C(T) $ is decreasing so that for $ T>1 $ we have
\[
\left \|\sum_{n=1}^N a_n e^{-\lambda_n t}\right \|_{L^{2}(0,T)}\leq
\left \|\sum_{n=1}^N a_n e^{-\lambda_n t}\right \|_{L^{2}(0,+\infty)}\leq C(1) \left \|\sum_{n=1}^N a_n e^{-\lambda_n t}\right \|_{L^{2}(0,T)}.
 \]
Hence we have
\[
\| P_T\|\leq 1\,,\qquad \|P^{-1}_T\|\leq C(1)\,.
 \]
 This implies, for $ T>1 $:
 \[
 |\Psi_n| _{L^2(0,T)}\leq C(1) |\Psi_n| _{L^2(0,+\infty)}\leq C(1) K(\ZEP) e^{\ZEP\zl_n}
  \]
as wanted.
\end{proof}

Finally, we cite~\cite{TENENBAUM-TUCSNAK} for different conditions on sequences of exponentials, which lead to (delicate) estimates on the solution of the corresponding moment problem.

 \subsection{NCVE for delay systems}

Now we  discuss a controlled delay system   also considered in \cite{PZ}. In this case a direct computation of the control energy  at time $ T $  as it is done in the previous example seems to be difficult (see the explanation below).  On the other hand, it is possible to apply
Corollary~\ref{coro:rema:nounif} to deduce that the system is NCVE.

Let us consider a retarded system with state delays,
\begin{equation}\ZLA{eq:esedelaygener}
\dot x=\sum_{k=0}^M A_k x(t- k\tau) +B u(t)
 \end{equation}
 where $ \tau>0 $,  $ x\in\mathbb{ R}^n $, $ A_i $ are $ n\times n $ constant matrix  and $ u\in {\mathbb R}^m $ (so that the constant matrix $ B $ is $ n\times m $). We introduce $ H=M\tau $.

 Eq.~(\ref{eq:esedelaygener}) is a model of a semigroup system in $ M^2={\mathbb  R}^n\times L^2(-H,0;{\mathbb R}^n) $, the state of the system being the couple $ (x(t),x(t-s) )$, with $ s\in [-H,0] $. See~\cite{BD1} for details.

 It turns out that:
  \begin{itemize}
  \item the semigroup is compact for $ t>H $, so that we are in the framework of Corollary~\ref{coro:rema:nounif}.
  \item the spectrum of the generator is not empty (it might be finite in special cases) and  (see~\cite{HALE}) its elements are the zeros of the holomorphic function
 \begin{equation}\ZLA{eq:degliautovRITARDO}
  \det\left [
\zl I-\sum_{k=0}^M A_k e^{-\zl  k\tau}
  \right ]
  \end{equation}
   ($ I $ is the $ n\times n $ identity matrix);
   \item the system is null controllable if and only if
 \begin{equation}\ZLA{eq:CondiCONTROLLRItardo}
   {\rm rank}\, \left [
   \begin{array}{cc}
   \zl I-\sum_{k=0}^M A_k e^{-\zl  k\tau}  &B
   \end{array}\right]=n
   \end{equation}
    for every $ \zl $ (see~\cite{OlbrPAND}).
  \end{itemize}
 So, we can state that this system is NCVE when condition~(\ref{eq:CondiCONTROLLRItardo}) holds and   the holomorphic function in~(\ref{eq:degliautovRITARDO}) has no zero with \emph{positive} real part.

Controllability can often be easily checked while conditions for nonpositivity of the real parts of eigenvalues have been widely studied (see for example~\cite{BellCOOK}).

Null controllability can be reduced to a moment problem of course, but arguments as those in
Section~\ref{sect:ExplEstimNCVE} for the computation of the control energy  seems difficult to apply, since now the eigenvalues are distributed on a (finite number) of sequences, each one of which has the following asymptotics:
 \[
 \zl_n=x_n+iy_n\,,\qquad x_n=m\big(\alpha -\log (\beta m) \big)+{\rm o}(1)\,,\quad y=\beta m+{\rm o}(1)
  \]
 ($\alpha $ and $\beta$  are suitable constants, see~\cite[Theorem~12.8]{BellCOOK}).

 The sequence $ \{-\zl_n\} $ does not satisfy neither the conditions in~\cite[Lemma~3.1]{F-R} neither those in~\cite[Lemma~3.1]{CARA}. Furthermore, since the retarded systems have smoothing property,  the sequence of the exponentials $ \{e^{-\zl_n t}\}$   is not a Riesz sequence  in $ L^2(0,T)$. So, it seems difficult in this case to compute explicitly the energy of the control at time $ T $.

 As a simple specific example we consider
 \[
 \dot x=-ax(t)- b x(t-\tau)+u(t)
  \]
  (which is clearly null controllable). Using~\cite[page 135]{HALE} we see that this system is NCVE for every value of $ r $ if and only if  $|b|\leq a$.
   For a specific $ r>0 $ the set of the $ (a,b) $
  plane in which NCVE holds is  represented in~\cite[Fig. 5.1]{HALE}.

\appendix

\section{\label{AppBASICsetting} The basic setting  for boundary control}

Here we present known facts about boundary control   which are
explained in~\cite[Sections~2.9 and 2.10]{TW}. Useful references are
~\cite[Chapter~2]{gold}, ~\cite[Chapter 3]{BD1} and~\cite[Chapter
1]{LT1}. However, it seems to us that detailed arguments are not completely  presented in standard control references and so we write  this appendix
 for the reader's convenience.

 A warning is needed: terms and some settings change  in
different books. For example~\cite{Kato} uses the same term,
adjoint, for Banach space and Hilbert space adjoints, while it is
convenient for us to use different terms. More important, the dual
spaces and the Banach space adjoints are defined in terms either of
linear forms or sesquilinear forms. The use of sesquilinear forms as
in~\cite{Kato,TW} is the most convenient for us.

We must introduce  few notations. As before, $
\langle\cdot,\cdot\rangle $ will be used to denote the inner product
in  Hilbert spaces (if needed, the spaces are specified with an
index; no index is present for the inner product in $ H $).

\ If $ V $ is a complex Banach space (possibly Hilbert),
$ V' $ denotes its topological dual (the Banach space of the
continuous {\em linear \/} functionals defined on $ V $).
Thus, if
$ \omega\in V' $ we can compute $ \omega(v) $ for every $ v\in V $.
We shall use the notation
\[
 {~}_{V'}(\omega,v)_{V}
 \]
 in order to denote the \emph{ sesquilinear}
  pairing of $ V $ and $ V' $, i.e.,
 $ {~}_{V'}(\omega,v)_{V}  $ is antilinear in $ \omega $ and linear in
 $ v $.

  To give an example, we note that the concrete  spaces encountered in
control theory are complexification of spaces of real functions;
i.e., if $ V_R $ is a linear space over $\mathbb R$,  the elements
of the corresponding complexified space $ V $ have the form
\[
v=f+ig\,,\qquad f\,,\ g\in V_R\,.
 \]
 The space $ V $ is a linear space on
$ \mathbb C $ and  it is simple to construct  sesquilinear forms on
$ V $, using  elements of $ V '$: let $ \omega$  a complex valued
linear functional on $ V $, i.e., $ \omega\in V' $.   The associated
sesquilinear form  is
\[
_{V'}(\omega,v)_V=  _{V'}(\omega,f+ig)_V=\overline{\omega(f-ig)}\,.
 \]

 We shall use both the Hilbert space adjoint and the dual of
an operator in the sense of Banach spaces. The Hilbert space adjoint
is defined for densely defined operators $ A $ by
\[
\langle Ax,y\rangle=\langle x,A^*y\rangle \qquad \forall x\in {\rm dom}\,A\,,\quad \forall y\in {\rm dom}\,A^*
 \]
(this equality implicitly defines $ {\rm dom}\,A^* $ as the set of those $ y\in H $ such that
$ x\to \langle Ax,y\rangle $ is continuous).

This implies in particular that
\[
\rho(A^*)=\overline{\rho(A)}\,.
 \]

The operator $ A^*  $ is closed and if  $ A $ is (densely defined)
closed   then $  A^* $ has dense domain too.

The Banach  space dual  of an operator   $ A$: ${\rm dom}\, A
\subset V \to W$ (here $V $ and $W$ are Banach spaces) will be
denoted $ A' $. It is a linear operator from $ W' $ to $ V' $. It is
(uniquely) defined for densely defined operators $ A $ and
\[
{\rm dom}\,A'=\{ \omega\in W'\,:\ \; v\mapsto {~}_{W'}(\omega,Av)_{W}\  \mbox{is continuous.}\ \}.
 \]
 By definition,
 \[
 {~}_{V'}(A'\omega,v)_{V}= {~}_{W'} (\omega,Av)_{W}\,.
  \]
 Sesquilinearity of the pairing implies that
  \[ \rho(A')=\overline{\rho(A) }\]
(see~\cite[pg. 184]{Kato}). Hence, the conjugate of multiplication by
$ \lambda  $ is multiplication by $ \bar\lambda $ (if instead the
conjugate is defined in terms of bilinear forms then the resolvent
is not changed).
 Moreover,
 $ A' $ has dense
domain if $ A $ has dense domain and it is closed, provided that $ V
$ is reflexive, in particular if it is a Hilbert space.

If $ W=V $ and if $ A $ is the infinitesimal generator of a $ C_0 $
semigroup on $ V $ then it might be that $ A' $ is not a generator
on $ V' $. It happens that $ A' $ is the infinitesimal generator of
a $ C_0 $-semigroup on $ V' $ if $ V $ is  reflexive, in particular
if it is a Hilbert space. In this case $ e^{A' t}=\left (e^{A
t}\right) ' $.
    As for the Hilbert space adjoint $A^*$, it generates $ e^{A^* t} $ (see~
\cite[Section 1.10]{Pa}).

  With these notations and preliminary information, we can now give the details of the setting used in the analysis of boundary control systems.

\subsection{The operators $ A $ and
$ \A=(  A^*)^{'}$}

  Let $A$ be the generator of a strongly continuous
semigroup $e^{At}$ on a complex Hilbert space $H$ with inner product
$\langle \cdot , \cdot\rangle $ and norm $|\cdot |$.

We shall identify its topological dual $ H' $ with $ H $ using the Riesz isomorphism, which we denote $ R $:   $  H \to H'$, defined as:
\[
 (Rv)(h)={~}_{H'}(Rv,h)_{H} = \langle h, v\rangle.
 \]
In practice, $ R $ is not explicitly written, hidden behind the
equality $ H=H' $ but in this appendix the distinction is needed for
clarity.

Using the Riesz map $ R$: $ H\mapsto H' $ and the definition of $ A' $,
we see that $ {\rm dom}\,A'=R\left ({\rm dom}\,A^*\right ) $. In fact
\[
_{H'}(  Rh,Ak  )_{H}=\langle Ak,h\rangle
 \]
and the right hand side is a continuous function of $ k $ if and only if the same holds for the left side.

For every $ h\in {\rm dom}\,A^* $   and every
 $ k \in {\rm dom}\, A $ we have:
\[
{~}_{H'}\left (  A'Rh ,k
\right )_{H}
 = {~}_{H'}\left ( R\left (R^{-1}A'Rh\right ),k
\right )_{H}=\langle k,R^{-1}A'Rh\rangle\,.
 \]
The definition of $ A' $ is
\[
{~}_{H'}\left (  A'Rh ,k
\right )_{H}={~}_{H'}\left (   Rh ,Ak
\right )_{H}=\langle k,A^*h\rangle\,.
 \]
 Hence (see~\cite[Sect.~II.7]{gold}) we have
\begin{equation}
\label{relaAeAPRIMO}
 A^*=R^{-1}A'R\,.
\end{equation}
  The same relation holds for the semigroups
  \[
  e^{A^*t}=R^{-1}e^{A' t}R\,.
   \]

  In the sequel we denote by $V$  the Hilbert space
${\rm dom}\,A^*$, with inner product $\langle h, v\rangle_{*}$ $ = \langle h, v\rangle $
$+ \langle A^*h, A^*v\rangle $, $h,v \in \text{dom}\,A^*.$ We have
\begin{align} \label{vv}
{\rm dom}\,A^*= V \subset H \stackrel{R}{\simeq} H' \subset V'
\end{align}
with dense and continuous injections.

Let $ {j \,}  $ be the injection of $ V $ into $ H $, $ {j \,}v=v\in
H $, for $v \in V$. Then, the definition of $ {j \,}' : H' \to V'$
is
\[
\langle {j \,}v,h\rangle={~}_{H'}(Rh,{j \,}v)_{H}={~}_{V'}({j
\,}'Rh, v)_{V}
 \]
and this shows  that $ {j \,}' Rh$ is the restriction of $ Rh $
(acting on $ H $) to the subspace $ V \subset H$.

As $A^* \in {\mathcal L}(V,
H)$ we have
$(A^*)' : H'   \to V'$ belongs to ${\mathcal L}(H', V')$.

We denote $(A^*)'$ by $\A$, so that
 $\text{dom}\,\A =H'$ (or, as usually written when $ H $ and $ H' $ are identified,
  $\text{dom}\,\A =H$).
  The crucial property used in control theory is
  expressed by stating that
 {\it $\A$ extends $A$.\/} The precise statement is:
 \begin{Lemma}
If $ x\in {\rm dom} \, A $ then we have:
\[
Ax=R^{-1}({j \,}')^{-1}\A \, Rx \,.
 \]
\end{Lemma}
\begin{proof}
  Indeed, if $x \in \text{dom}\, A$, $v \in V =
\text{dom}\, A^*$, then
\begin{align*}
_{V'}( \A \,  Rx, v)_{V}= {~}_{V'}( (A^*)' \, Rx, v)_{V}
 = { _{H'}(} Rx, A^* v)_{H}
 = \langle A^*v, x\rangle  \\
= \langle v, Ax \rangle = { _{H'}(}R Ax, {j \,}v)_{H}={~}_{V'}({j
\,}'RAx,v)_{V}
\end{align*}
and so $ \A R={j \,}'RA $. When $ {j \,}' $ and $ R $ are not
explicitly written, as usual, we get $\A x = A x $.
\end{proof}

The second property that we want to prove is that $ V' $ is an
extrapolation space generated by $ A $. This means that we can see $
V' $ as the completion of $ H $ when endowed with the norm
$|(\lambda I - A)^{-1} \cdot |$, for any $ \lambda \in\rho(A) $. In
order to see this, we fix any $ \lambda \in \rho(A) $ and   we prove
that $ |\cdot| _{V'} $ restricted to  $ H $ is equivalent to
$|(\lambda I - A)^{-1}\cdot |$,
 i.e., we prove:
 \begin{Lemma}
On $ H $, the norms of $ V' $ (more precisely, $ h\mapsto |{j \,}'
Rh|_{V'} $) and the norm $ |(\lambda I-A)^{-1} \cdot | $ are
equivalent.
\end{Lemma}
\begin{proof}
Let $ I $ denote the identity in $ H $ and also on $ V $.

Let $ \lambda \in\rho(A)  $ and $h \in H$.
 In order to compute $ |{j \,}' Rh|_{V'} $, we proceed as follows (recall
 that  $ {j \,}' Rh$ is the restriction of $ Rh $
 to  $ V $).
\begin{eqnarray*}
&&|{j \,}' Rh|_{V'}=\sup _{|v|_V \le 1}  |{}_{V'}({j \,}' Rh  , v
)_V |= \sup _{|v|_V \le 1}|\langle v, h \rangle |
\\
&& =\sup_{|v|_V \le 1} |\langle  \left( \bar \lambda { I}-{
A^*}\right )^{-1} \left ( \bar \lambda { I}-{ A^*}\right )
 v , h \rangle|
\\
&&  =\sup_{|v|_V \le 1} |\langle  \left ( \bar \lambda { I}-{
A^*}\right )
 v , \left(  \lambda { I}-{
A}\right )^{-1}  h \rangle|
 \le C \left |\left
(\lambda I-A\right )^{-1}h\right |,
\end{eqnarray*}
 with $C = \sup_{|v|_V \le 1}  |(\bar \lambda
I-A^*)v|.$ On the other hand,
$$
\left |\left (\lambda I-A\right )^{-1}h\right |
 \le
 C_0 |Rh|_{H'} = C_0   |({j \,} ')^{-1}  {j \,} ' Rh|_{H'}
\le C_1  | {j \,} ' Rh|_{V'}.
$$
The proof is complete.
\end{proof}

\subsection {The parabolic case and
 fractional powers of
 $(\omega I- A^{*})$
 }

 Our goal here is   to show that if $A$ (and so $A^*$)
 {\it generates a holomorphic
 semigroup}  $e^{At}$ (respectively $e^{A^* t}$) and in addition
\begin{align} \label{dom2}
B \in {\mathcal L} \left(U , \left(\text {dom}\,\left(\omega -A^*\right)^{\gamma} \right)'\right),
\end{align}
for some $\gamma \in [0,1)$ and $\omega \in \rho(A)$,
 then we have  the crucial estimate \eqref{DiseSINGOLARE},
 i.e.,
$$
 \| e^{\A t} B \|_{{\mathcal L}(U,H)} \le \frac {Me^{\omega_1 t}} {t^{\gamma}}, \;\;\; t
>0.
$$
First  recall that, since $A$ is holomorphic,
there always exists
  $\omega \in \rho (A)$ such that, for any $\gamma
\in (0,1)$, $(\omega I - A^{*})^{\gamma}$ is a well defined
 closed operator with domain $V_{\gamma} \subset H$
 (cf.~\cite[Section 2.6]{Pa}).
Then   note that   $\A$ generates
  the dual semigroup $e^{\A t} $ which is holomorphic
   on $V'$.

Arguing as in \eqref{vv}, we have
$$
 V_{\gamma} \subset H \simeq H' \subset V'_{\gamma}
$$
and we denote by $\A_{\gamma}$ the operator $[(\omega -
A^{*})^{\gamma}]'$ $: H \to  V'_{\gamma}.$
 Since  $B \in {\mathcal L} (U,
 V'_{\gamma})$,
  we may consider $B: U \to V'$, since
 $ V'_{\gamma} \subset V'$ with dense and continuous
 injections. Thus  we also have
   $B \in {\mathcal L} (U, V')$.

Moreover, since $B  \in {\mathcal L} (U,
 V'_{\gamma})$, we have
  $B'   \in {\mathcal L} (V_{\gamma}, U')$
 and so, for $ t>0 $:
$$
 \| B' e^{A^*t} \|_{{\mathcal L}(H,U')}
  =  \| B' (\omega I - A^{*})^{-\gamma} \,
   (\omega I - A^{*})^{\gamma} e^{A^*t} \|_{{\mathcal L}(H,U')}
$$
\begin{align} \label{fit}\le   \| B' (\omega I - A^{*})^{-\gamma} \|_{{\mathcal L}(H,U')}
  \,  \|  (\omega I - A^{*})^{\gamma} e^{A^*t}  \|_{{\mathcal L}(H,H)}
 \le \frac{Me^{\omega_1 t}}{t^{\gamma}}, \;\; t>0
\end{align}
(in the last line we have used a well known estimate for holomorphic
semigroups). Next   {\it we  compute the dual operator of $B'
e^{A^*t}$
 and show that it is $ R^{-1} e^{{\mathcal A} t}B $,}
 usually
written as $ e^{{\mathcal A} t}B  $ when $ H $ and $ H' $ are
identified.

 We have, for any $x \in H$, $u \in U$, $t>0$,
 $$
  { _{U'}(} B' e^{A^* t} x, u)_{U}=
  { _{V_{\gamma}}(}  e^{A^* t} x, B u)_{ V_{\gamma}'}
 =  { _{V}(}  e^{A^* t} x, B u)_{ V'},
$$
since $Bu \in V_{\gamma}' \subset V'$ (recall that $V=
\text{dom}\,A^{*})$ and $e^{A^* t } x \in V$, $t>0$. It follows that
$$
{ _{U'}(} B' e^{A^* t} x, u)_{U} =
 { _{H'}(}  e^{\A t } Bu,x)_{H}=\langle x,R^{-1}e^{\A t}Bu\rangle\,.
$$
 and so the claim follows.

The previous assertion
 implies the identity
\begin{align} \label{fie}
  \| B' e^{A^*t} \|_{{\mathcal L}(H,U')} =
  \| R^{-1}e^{\A t} B \|_{{\mathcal L}(U,H)},\;\; t>0,
\end{align}
which together with \eqref{fit} implies the estimate
\eqref{DiseSINGOLARE}
which has been used in Subsection 1.1.

\vskip 4mm
\noindent
\textbf{Acknowledgement.}
The authors would like to thank the  referees
 for their
 useful remarks and suggestions
 which helped to  improve the paper.

\end{document}